\documentclass[12pt,reqno]{amsart}

\usepackage{amsmath,amsthm,comment,fullpage}
\usepackage{upgreek}

\usepackage{times}
\usepackage[T1]{fontenc}
\usepackage{mathrsfs}
\usepackage{latexsym}
\usepackage[dvips]{graphics}
\usepackage{epsfig}

\usepackage{amsmath,amsfonts,amsthm,amssymb,amscd}
\usepackage{color}
\usepackage{hyperref}
\usepackage{breakurl}
\usepackage{url}

\newcommand{\bburl}[1]{\textcolor{blue}{\url{#1}}}
\def\legendre#1#2{\left(\frac{#1}{#2}\right)}
\def\cross{\times}

\numberwithin{equation}{section}

\newtheorem{thm}{Theorem}[section]

\newtheorem{lem}[thm]{Lemma}
\newtheorem{prop}[thm]{Proposition}

\theoremstyle{plain}

\newtheorem{definition}[thm]{Definition}

\newtheorem{lemma}[thm]{Lemma}
\newtheorem{proposition}[thm]{Proposition}
\newtheorem{theorem}[thm]{Theorem}

\newtheorem{remark}[thm]{Remark}

\newcommand\bal{\begin{align}}
\newcommand\eal{\end{align}}
\newcommand\be{\begin{equation}}
\newcommand\ee{\end{equation}}
\newcommand\bea{\begin{eqnarray}}
\newcommand\eea{\end{eqnarray}}
\newcommand\bi{\begin{itemize}}
\newcommand\ei{\end{itemize}}
\newcommand\ben{\begin{enumerate}}
\newcommand\een{\end{enumerate}}
\newcommand\bc{\begin{center}}
\newcommand\ec{\end{center}}
\newcommand\ba{\begin{array}}
\newcommand\ea{\end{array}}

\newcommand{\R}{\ensuremath{\mathbb{R}}}

\newcommand{\Z}{\ensuremath{\mathbb{Z}}}
\newcommand{\Q}{\mathbb{Q}}
\newcommand{\N}{\mathbb{N}}
\newcommand{\F}{\mathbb{F}}

\newcommand\frakfamily{\usefont{U}{yfrak}{m}{n}}
\DeclareTextFontCommand{\textfrak}{\frakfamily}

\newtheorem{rek}[thm]{Remark}

\newcommand{\ncr}[2]{{#1 \choose #2}}

\newcommand{\hr}[1]{\href{#1}{\url{#1}}}

\newcommand\nc{\newcommand}
\nc{\on}{\operatorname}
\nc{\Hom}{\on{Hom}}
\nc{\wt}{\widetilde}
\nc{\kernel}{\text{ker}}
\nc{\image}{\text{Im}}
\nc{\sls}{\subsetneq ... \subsetneq}
\nc{\ssn}{\subsetneq}
\nc{\bull}{$\bullet \, \,$}
\nc{\short}[3]{0 \longrightarrow #1 \longrightarrow #2 \longrightarrow #3 \longrightarrow 0}
\nc{\pd}[2]{\frac{\partial #1}{\partial #2}}

\newcommand{\js}[1]{{#1\overwithdelims () p}}

\title{Lower-Order Biases in the Second Moment of Dirichlet Coefficients in Families of $L$-Functions}

\author{Megumi Asada}
\email{\textcolor{blue}{\href{mailto:maa2@williams.edu}{maa2@williams.edu}}}
\address{Department of Mathematics and Statistics, Williams College, Williamstown, MA 01267}

\author{Ryan C. Chen}
\email{\textcolor{blue}{\href{mailto:rcchen@mit.edu}{rcchen@mit.edu}}}
\address{Department of Mathematics, Massachussetts Institute of Technology, Cambridge, MA 02142}

\author{Eva Fourakis}
\email{\textcolor{blue}{\href{mailto:erf1@williams.edu}{erf1@williams.edu}}}
\address{Department of Mathematics and Statistics, Williams College, Williamstown, MA 01267}

\author{Yujin Hong Kim}
\email{\textcolor{blue}{\href{mailto:yujin.kim@cims.nyu.edu}{yujin.kim@cims.nyu.edu}}}
\address{Courant Institute of Mathematical Sciences, New York University, New York, NY 10012}

\author{Andrew Kwon}
\email{\textcolor{blue}{\href{mailto:akwon@andrew.cmu.edu}{akwon@andrew.cmu.edu}}}
\address{Department of Mathematics, Carnegie Mellon University, Pittsburgh, PA 15213}

\author{Jared Duker Lichtman}
\email{\textcolor{blue}{\href{mailto:jared.d.lichtman@gmail.com}{jared.d.lichtman@gmail.com}}}
\address{Department of Mathematics, Dartmouth College, Hanover, NH 03755}

\author{Blake Mackall}
\email{\textcolor{blue}{\href{mailto:bmackall60@gmail.com}{bmackall60@gmail.com}}}
\address{Department of Mathematics and Statistics, Williams College, Williamstown, MA 01267}

\author{Steven J. Miller}
\email{\textcolor{blue}{\href{mailto:sjm1@williams.edu}{sjm1@williams.edu}}}
\address{Department of Mathematics and Statistics, Williams College, Williamstown, MA 01267}

\author{Eric Winsor}
\email{\textcolor{blue}{\href{mailto:rcwnsr@umich.edu}{rcwnsr@umich.edu}}}
\address{Department of Mathematics, University of Michigan, Ann Arbor, MI 48109 }

\author{Karl Winsor}
\email{\textcolor{blue}{\href{mailto:krlwnsr@umich.edu}{krlwnsr@umich.edu}}}
\address{Department of Mathematics, University of Michigan, Ann Arbor, MI 48109 }

\author{Jianing Yang}
\email{\textcolor{blue}{\href{mailto:jyang@colby.edu}{jyang@colby.edu}}}
\address{Department of Mathematics, Colby College, Waterville, ME 04901 }

\author{Kevin Yang}
\email{\textcolor{blue}{\href{mailto:kyang95@stanford.edu}{kyang95@stanford.edu}}}
\address{Department of Mathematics, Harvard University, Cambridge, MA 02138 }

\thanks{This work was supported by NSF grants DMS 1347804, DMS1265673 and DMS1561945, Carnegie Mellon University, Princeton University, Williams College, the Eureka Program, the Finnerty Fund, and the Clare Boothe Luce Program of the Henry Luce Foundation. We thank the referee for numerous helpful comments, and Matija Kazalicki and Bartosz Naskrecki for sharing their preprints.}


\date{\today}

\begin{document}
\maketitle

\begin{abstract}
Let $\mathcal E: y^2 = x^3 + A(T)x + B(T)$ be a nontrivial one-parameter family of elliptic curves over $\mathbb{Q}(T)$, with $A(T), B(T) \in \mathbb Z(T)$, and consider the $k$\textsuperscript{th} moments $A_{k,\mathcal{E}}(p) := \sum_{t (p)} a_{\mathcal{E}_t}(p)^k$ of the Dirichlet coefficients
$a_{\mathcal{E}_t}(p) := p + 1 - |\mathcal{E}_t (\mathbb{F}_p)|$. Rosen and Silverman proved a conjecture of Nagao relating the first moment
$A_{1,\mathcal{E}}(p)$ to the rank of the family over $\mathbb{Q}(T)$, and Michel proved that if $j(T)$ is not constant then the second moment is
equal to $A_{2,\mathcal{E}}(p) = p^2 + O(p^{3/2})$. Cohomological arguments show that the lower order terms are of sizes $p^{3/2}, p, p^{1/2}$, and $1$. In every case we are able to analyze in closed form, the largest lower order term in the second moment expansion that does not average to zero is on average negative, though numerics suggest this may fail for families of moderate rank. We prove this Bias Conjecture for several large classes of families, including families with rank, complex multiplication, and constant $j(T)$-invariant. We also study the analogous Bias Conjecture for families of Dirichlet characters, holomorphic forms on GL$(2)/\mathbb{Q}$, and their symmetric powers and Rankin-Selberg convolutions. We identify all lower order terms in large classes of families, shedding light on the arithmetic objects controlling these terms. The negative bias in these lower order terms has implications toward the excess rank conjecture and the behavior of zeros near the central point in families of $L$-functions.
\end{abstract}

\tableofcontents


\section{Introduction}

The genesis of this paper are some computations on moments of the Dirichlet coefficients of the $L$-functions of elliptic curves by Miller in his thesis \cite{Mi1, Mi2} and expanded in \cite{Mi3}. The main purpose of that work was to verify the Katz-Sarnak Density Conjecture \cite{KaSa1, KaSa2} for families of elliptic curves; in other words, that in the limit as the conductors tend to infinity the behavior of zeros near the central point in families of elliptic curve $L$-functions agree with the scaling limits of eigenvalues near 1 of orthogonal groups. In that and related work on numerous other families of $L$-functions (see \cite{MMRT-BW} for an extensive discussion and survey of the literature  and an expanded version of the argument below), the main term of number theory and random matrix theory agree; the lower order terms on the two sides, however, often differ, and it is in these lower order terms that the arithmetic of the family emerges. This is similar to the central limit theorem, which we expand on later.


In all families studied to date which can be computed in closed form a bias has been observed in these lower order terms; the purpose of this work is to describe these results, extensively investigate this phenomenon in other families, discuss the implications such a bias has in number theory, and hopefully motivate others to provide a theoretical understanding for this phenomenon (explicitly, as there is a cohomological interpretation of these elliptic curve moments, could such arguments give a proof for the bias?). Two of the most important applications are to the distribution of ranks in families, and lower order terms in the $n$-level densities; in particular, the presence of a negative bias leads to a slight increase in bounds on the average rank in families, as well as surfacing in lower order terms of the $n$-level densities of zeros of $L$-functions. The latter is particularly important, as the main term of these statistics agree with those from random matrix theory; the arithmetic of the family is absent in the main term but arises in the lower order terms,  and breaks the universality of behavior. Thus we first describe the general framework (see \cite{IK} for proofs and additional details), then describe the families studied and state our results.

\subsection{Bias in Families of \texorpdfstring{$L$}{L}-Functions}

We quickly review some needed properties of $L$-functions; see for example \cite{IK,RudSa}. Let $\pi$ be a cuspidal automorphic representation on ${\rm GL}_n$, and let $Q_\pi>0$ be the analytic conductor of the associated $L$-function \be L(s,\pi)\ =\ \sum_{m=1}^\infty \frac{\lambda_\pi(m)}{m^s}, \ \ \ \frac{L'(s,\pi)}{L(s,\pi)} \ = \ -\sum_{m=1}^\infty \frac{\Lambda(m) a_\pi(m)}{m^s}, \ee where $\Lambda(m) = \log p$ if $m = p^k$ for a prime $p$, and 0 otherwise. Assuming the Generalized Riemann Hypothesis (GRH) the non-trivial zeros are of the form $1/2 + i \gamma_{\pi,j}$ with $\gamma_{\pi,j}$  real. Let $\{\alpha_{\pi,i}(p)\}_{i=1}^n$ be the Satake parameters of $L(s,\pi)$. We have the Euler product \be L(s,\pi)\ =\ \sum_{m=1}^\infty \frac{\lambda_\pi(m)}{m^s}\ =\ \prod_p \prod_{i=1}^n \left(1 -\alpha_{\pi,i}(p) p^{-s}\right)^{-1}, \ee and \be a_\pi(p^\nu) \ = \ \sum_{i=1}^n \alpha_{\pi,i}(p)^\nu; \ee thus the $p^\nu$\textsuperscript{{\rm th}} coefficient of $-L'(s,\pi)/L(s,\pi)$ is the $\nu$\textsuperscript{{\rm th}} moment of the Satake parameters (up to a factor of $\log p$). The explicit formula (see for example \cite{ILS,RudSa}), applied to a given $L(s,\pi)$ and then averaged over a finite family  $\mathcal{F}_N$, yields the 1-level density \begin{eqnarray}
  D_{1,\mathcal{F}_N}(\phi) & \ := \ & \frac1{|\mathcal{F}_N|} \sum_{\pi\in\mathcal{F}_N}\sum_j \phi\left(\gamma_{\pi,j}\frac{\log Q_\pi}{2\pi}\right)   \nonumber\\ & \ =\ &   \widehat{\phi}(0) -   2 \frac1{|\mathcal{F}_N|} \sum_{\pi\in\mathcal{F}_N}\sum_{p} \sum_{\nu = 1}^\infty\widehat{\phi}\left(\frac{\nu\log p}{\log Q_\pi}\right)  \frac{a_\pi(p^\nu)\log p}{p^{\nu/2}\log Q_\pi} \nonumber\\ & & \ \ \ + \ O\left(\frac{\log\log Q_\pi}{\log Q_\pi}\right),
\end{eqnarray} where $\phi$ is an even Schwarz test function with compactly supported Fourier transform $\widehat{\phi}$ and $N$ is some parameter such that as $N$ increases, the analytic conductors increase.\footnote{Note the 1-level density is well-defined even if GRH fails, though if there are zeros off the line then we lose the spectral interpretation of the zeros. If we adjust the rescaling of the zeros slightly we can remove the big-Oh error term, but its presence does not matter for calculating the main term, and is only important when we look at the 2 or higher level densities.} The Katz-Sarnak Density Conjecture states that as $N\to\infty$ the 1-level density converges to that of a classical compact group. This has been verified for many families for test functions whose Fourier transforms have suitably restricted support; see \cite{MMRT-BW} for a list of many of these families, as well as a summary of the techniques used in the proofs.

In many situations, such as families of Dirichlet characters or cuspidal newforms of a given level and weight, the analytic conductors in our family are constant; thus $Q_\pi = Q_N$ say. For other families such as elliptic curves this fails, and one must either do sieving and additional work, or instead normalize by the average log-conductor; while this is satisfactory for the 1-level density it introduces problems for the general $n$-level density (see \cite{Mi1, Mi2} for details and a resolution). Thus in calculating the 1-level density we can often push the sum over our family $\mathcal{F}_N$ through the test function and reduce the analysis to averages of the moments of the Satake parameters. In all families studied to date we have sufficient decay in the $\lambda_\pi(p^\nu)$'s so that the sum over primes with $\nu \ge 3$ converges; this is known for many families, and follows from the Ramanujan conjectures in general.\footnote{The Satake parameters $|\alpha_{\pi,i}|$ are bounded by $p^\delta$ for some $\delta$; conjecturally $\delta = 0$. There has been significant progress towards these bounds with some $\delta < 1/2$; see \cite{K,KSa}. Any $\delta < 1/6$ implies the $\nu \ge 3$ terms do not contribute to the main term.} Thus determining the 1-level density, up to lower order terms, is equivalent to analyzing the $N\to\infty$ limits of \bea\label{eq:S1S2} S_1(\mathcal{F}_N) & \ := \  & -2 \sum_{p}
\widehat \phi\left(\ \frac{\log p}{\log Q_N}\right) \frac{\log p}{\sqrt{p}
\log Q_N}\ \left[\frac{1}{|\mathcal{F}_N|} \sum_{\pi \in \mathcal{F}_N}
a_\pi(p)\ \right]
\nonumber\\ S_2(\mathcal{F}_N) & \ := \ & -2 \sum_{p}\widehat \phi\left(2\frac{\log p}{\log
Q_N}\right)\ \frac{\log p}{p \log Q_N}\ \ \left[
\frac{1}{|\mathcal{F}_N|} \sum_{\pi \in \mathcal{F}_N} a_\pi(p^2)\right].
\eea As
\be a_\pi(p^\nu) \ = \ \alpha_{\pi,1}(p)^\nu + \cdots
+ \alpha_{\pi,n}(p)^\nu, \ee we see that only the first two moments of the Satake parameters enter the calculation. The sum over the remaining powers, \be S_\nu(\mathcal{F}_N)  \ := \  -2  \sum_{\nu = 3}^\infty \sum_p \widehat \phi\left(\nu\frac{\log p}{\log Q_N}\right)\ \frac{\log p}{p^{\nu/2} \log Q_N}\ \ \left[
\frac{1}{|\mathcal{F}_N|} \sum_{\pi \in \mathcal{F}_N} a_\pi(p^\nu)\right], \ee is $O(1/\log Q_N)$ under the Ramanujan Conjectures. Thus the main term of the limiting behavior is controlled by the main terms of the first two moments of the Satake parameters (see Remark \ref{rek:firsttwoterms} for more on this); the higher moments (and the lower order terms in the first two moments) affect the rate of convergence to the random matrix theory limits. The goal of this work is to explore these lower order terms in a variety of families, and analyze the consequences for number theory.

As remarked above, the non-universality in the lower order terms is similar to that of the Central Limit Theorem. Given any nice density, one can renormalize it to have mean zero and variance one. The universality of the Central Limit Theorem is due to the fact that the higher moments of the density, which is where the shape emerges, only surface as lower order terms in the analysis, affecting only the \emph{rate} of convergence to the Gaussian. The situation is thus very similar in families of $L$-functions, where both the third and higher moments of the Satake parameters, as well as lower order terms in the first and second moments, break universality and lead to lower order terms where arithmetic lives. These arguments show the importance of understanding the lower order terms, and is the motivation for this paper.

\begin{rek}\label{rek:firsttwoterms} We briefly comment on the first two moments, i.e., the sums in \eqref{eq:S1S2}. For test functions with small support, $S_1(\mathcal{F}_N)$ contributes zero to the main term in all families investigated to date save for families of elliptic curves with rank $r$, where the main term of its contribution is $-r$ (if the support is large it can contribute; see \cite{ILS} and the family of cuspidal newforms when the support exceeds $[-1, 1]]$). Additionally, $S_2(\mathcal{F}_N)$ contributes $-c_{\mathcal{F}} \phi(0)/2$, and the family of $L$-functions has unitary, symplectic or orthogonal symmetry depending on whether or not the \emph{family symmetry constant} $c_{\mathcal{F}}$ equals 0, 1 or -1. Further, in many cases it can be shown the symmetry of the Rankin-Selberg convolution of two families, $c_{\mathcal{F}\times\mathcal{G}}$, equals the product of the symmetries of the families. See \cite{DuMi, SaShTe, ShTe} for more on determining the symmetry of a family, and \cite{RudSa} for consequences of the second moment in universal behavior in another statistic, the $n$-level correlations.
\end{rek}

\subsection{Bias in Elliptic Curve Families}
\label{sec.define.Dirichlet}

As the initial impetus for this work came from families of elliptic curves, we start with a description of those results and then move to other families. Consider the elliptic surface $\mathcal{E}$ given by $y^2 = x^3 + A(T)x + B(T)$ over $\Q(T)$ for $A(T), B(T) \in \Z[T]$; for almost all $t \in \Z$ the specialization $\mathcal{E}_t$ obtained by setting $T=t$ is an elliptic curve.
Let $a_{\mathcal{E}_t}(p)$ denote $p$ minus the number of solutions to $y^2 \equiv x^3 + A(t) x + B(t) \pmod p$; this quantity is known as the $p^{\text{th}}$ \textbf{Dirichlet coefficient} of $\mathcal{E}_t$, as such coefficients appear in the series expansion for the associated $L$-function.
Set
\be A_{r,\mathcal{E}}(p) \ = \ \sum_{t\bmod p} a_{\mathcal{E}_t}^r(p), \ee so $A_{r,\mathcal{E}}(p)/p$ is the $r$\textsuperscript{th} moment.\footnote{We need to divide the $p$ sum to have a moment, as we are averaging over $p$ terms; some works include this division in the definition, others have it separate. We choose not to divide by $p$ so that our sums are integer polynomials in $p$.} By work of Nagao, Rosen and Silverman \cite{Na, RoSi} the first moment is related to the rank of the elliptic surface (it is a theorem if the surface is rational\footnote{An elliptic surface $y^2 = x^3 + A(T)x + B(T)$ is rational if and only if one of the following is true: (1) $0 < \max(\deg A, 2\deg B) < 12$; (2) $3\deg A = 2\deg B = 12$ and ${\rm ord}_{T=0} T^{12} \Delta(T^{-1}) = 0$.}, and conjectural in general):
\be \lim_{X\to\infty} \frac1{X} \sum_{p \le X} -\frac{A_{1,\mathcal{E}}(p)}{p} \log p \ = \ {\rm rank}~\mathcal{E}(\Q(T)). \ee
The main term of the second moment determines which classical compact group governs the behavior of zeros near the central point \cite{DuMi,Mi2,ShTe}, and by work of Michel \cite{Mic} we have \be A_{2,\mathcal{E}}(p) \ = \ p^2 + O(p^{3/2}) \ee if $j(T)$ is non-constant for the family.

The interesting observation in \cite{Mi1, Mi3} is that, in every family of elliptic curves investigated, the largest lower order term in $A_{2,\mathcal{E}}(p)$ which did not average to zero had a negative average. While there were some families with a lower order term of size $p^{3/2}$ (and thus Michel's bound is sharp), in all those families such terms were on average zero. There were many families with a lower order term of size $-m_{\mathcal{E}} p$ for some $m_{\mathcal{E}} > 0$. While these terms are too small to influence the main term, they yield corrections of size the reciprocal of the logarithm of the conductors (which is the natural spacing between zeros near the central point); explicitly, they increase the $1$-level  density by \begin{eqnarray} \frac{2m_{\mathcal{E}}}{\log Q} \sum_p \widehat{\phi}\left(2\frac{\log p}{\log Q}\right) \frac{\log p}{p^2},\end{eqnarray} and thus break the universality of the main terms and lead to lower order terms depending on the arithmetic of the family.

For small conductors one consequence of the negative bias $-m_{\mathcal{E}}p$ is to increase the bounds for the average rank in families. While the amount of the increase is too small to explain the entirety of the observed excess rank phenomenon (see \cite{DHKMS1, DHKMS2} for an explanation through the excised orthogonal ensemble), it \emph{is} in the right direction and it is fascinating that the bias is always observed to be in the same direction. Of course, the families investigated are very special, as they are ones where the Legendre sum can be computed exactly so that the second moment can be fully determined, and thus may not be truly representative; see Appendix \ref{sec:michellewu} for some families where we cannot compute the second moment in closed form but we do still observe the negative bias, and a discussion in Appendix \ref{sec:rogerwang} on some one-parameter families of moderate rank where there may be a positive bias. We present some of this evidence in Table \ref{table:ellcurvedata} (the calculations for most of these families can be found in \cite{Mi1, Mi3}; the new families can be handled using similar techniques).

\begin{table}
\begin{tabular}{lcl} \\
${\ \ \ \ \ \ \ \mbox{Family}  \ \ \ \ \ \ \ }$ & ${\ \ \ \ \ \ \
A_{1,\mathcal{E}}(p) \ \ \ \ \ \ \ }$ & ${\ \ \ \ \ \ \
A_{2,\mathcal{E}}(p) \ \ \ \ \ \ \ }$
\\ \hline

$y^2 = x^3 + Sx + T$ & $\ \ 0$    & $p^3 - p^2$ \\

$y^2 = x^3 + 2^4(-3)^3(9T+1)^2$ & $\ \ 0$ & $\Big\{ {2p^2 - 2p \ \
\ p
\equiv 2 \bmod 3\atop \ 0 \ \ \ \ \ \ \ \ \ \ p \equiv 1 \bmod 3 }$ \\

$y^2 = x^3 \pm 4(4T+2)x$ & $\ \ 0$ & $\Big\{ {2p^2 - 2p \ \ \  p
\equiv
1 \bmod 4 \atop \ 0 \ \ \ \ \ \ \ \ \ \ p \equiv 3 \bmod 4}$ \\

$y^2 = x^3 + (T+1)x^2 +Tx$ & $\ \ 0$ & $p^2 - 2p - 1$ \\

$y^2 = x^3 + x^2 + 2T+1$ & $\ \ 0$ & $p^2 - 2p - \js{-3}$ \\

$y^2 = x^3 + Tx^2 + 1$ & $-p$ & $p^2 -
n_{3,2,p}p - 1 + c_{3/2}(p)$ \\

$y^2 = x^3 - T^2x + T^2$ & $-2p$ & $p^2 - p - c_1(p) - c_0 p$ \\

$y^2 = x^3 - T^2x + T^4$ & $-2p$ & $p^2 - p - c_1(p) - c_0 p$ \\

$y^2 = x^3 + T x^2 - (T + 3) x + 1$ & $-2\delta_{1;4}(p) p$ & $p^2 - 4 \delta_{1;6}(p) p - 1$  \\

\ & & \\

\end{tabular}
\caption{\label{table:ellcurvedata} First and second moments for elliptic curve families, with $n_{3,2,p}$ the number of cube roots of $2$ modulo $p$, $c_0 p = \left[\js{-3} + \js{3}\right]$, $c_1(p) = \left[\sum_{x
\bmod p} \js{x^3-x}\right]^2$, $c_{3/2}(p) = p\sum_{x(p)} \js{4x^3+1}$, and $\delta_{a;m}(p) = 1$ if $p \equiv a \bmod m$ and otherwise is 0. The subscript on $c_k(p)$ indicates the power of $p$.}
\end{table}

\subsection{Outline}

In the present work we explore biases in the lower order terms of second moments in several different families of $L$-functions. A preliminary analysis of some elliptic curve families was reported in \cite{MMRW}; we provide additional proofs for some of the families explored there as well as some new ones, and also investigate numerous other natural families (Dirichlet $L$-functions, cuspidal newforms, and their convolutions). In \cite{HKLM} these investigations are extended to hyper-elliptic curves, and in Appendix \ref{sec:michellewu} we report on some recent extensions to other one-parameter families of elliptic curves, as well as some two-parameter families.

One challenge is to make sure we are comparing similar items in each case. In particular, what normalization should we use for the sums? For example, consider one-parameter elliptic curve families over $\Q(T)$ with $T \in [N, 2N]$. For each $p$, $p$ times the second moment is $\sum_{t (p)} a_{\mathcal{E}_t}(p)^2$ (we will have $\lfloor N/p\rfloor$ copies of this, and then an incomplete sum of at most $p$ terms); by Michel's work the main term is of size $p^2$ if $j(T)$ is not constant, and we observe lower order terms not averaging to zero of size $p$. The sizes here are due to our normalizations, where the elliptic curve $L$-function has critical strip $1/2 < {\rm Re}(s) < 3/2$ and functional equation $s \to 2-s$. To compare with other families of $L$-functions, where the critical strip is taken to be $0 < {\rm Re}(s) < 1$, we normalize and study $a_{\mathcal{E}_t}(p)/\sqrt{p}$; by Hasse's theorem each of these terms is at most $2$ in absolute value.

Thus our complete sums over $t \bmod p$ have a main term of size $p$ and the first term not averaging to zero is of size 1; however, we should also average over the family. For one-parameter families of elliptic curves we often look at $t \in [N, 2N]$ with $N\to\infty$; this gives us $\lfloor N/p\rfloor$ complete sums of $t \bmod p$ and one incomplete sum of size at most $p$. Thus we have \be \frac1{N} \sum_{t=N}^{2N} \frac{a_{\mathcal{E}_t}(p)^2}{p} \ = \ \frac{\lfloor N/p\rfloor}{N}  \sum_{t \bmod p} \frac{a_{\mathcal{E}_t}(p)^2}{p} + \frac1{N}\sideset{}{^\ast}\sum_t \frac{a_{\mathcal{E}_t}(p)^2}{p}, \ee where the asterisk denotes an incomplete sum of at most $p-1$ terms. Thus, as long as $p$ is significantly less than $N$, the incomplete sum is negligible relative to the first term. We will sum over $p \le X$ and divide by $\pi(X)$, the number of primes at most $X$, namely, \be\label{eq:ellcurvesecondmomentfullsum} \frac1{\pi(X)} \sum_{p \le X} \frac1{N} \sum_{t=N}^{2N} \frac{a_{\mathcal{E}_t}(p)^2}{p} \ =  \ \frac{1+o(1)}{\pi(X)} \sum_{p\le X} \sum_{t \bmod p} \frac{a_{\mathcal{E}_t}(p)^2}{p^2}. \ee
If the complete sum over $t \bmod p$ equals $p^2 - m_{\mathcal{E}}p$ then the $p^2$ factor yields the main term of 1 (recall we have $\lfloor N/p\rfloor$ complete sums) while the $m_{\mathcal{E}}$ yields the leading lower order term of size \be -\frac1{\pi(X)} \sum_{p \le X} \frac{m_{\mathcal{E}}}{p} \ = \ -\frac{m_{\mathcal{E}}\log\log X \log X}{X} \left(1 + o(1)\right),\ee  where we used standard results for the size of $\pi(X)$ and the sum of the reciprocals of the primes up to $X$.

\begin{rek} It's important to put the size of the main and leading error term in perspective. With this averaging, the small bias leads to a contribution which is barely detectable, tending to zero rapidly as the range of primes averaged over grows. This is quite reasonable, as the relative size of the bias to the main term, at each prime, is of size $1/p$, and this leads to a slowly growing sum. Note this behavior is very different than what happens when we look at the contribution of these lower order terms in the $n$-level densities and the excess rank investigations (this means looking at the average rank for forms in our family with parameter up to some specified value, and letting that value tend to infinity). The difference is due to how we weigh the sums. For the $n$-level densities and the rank, we are not dividing by the number of primes and are weighting each term by $(\log p)/p$. Thus the main and leading lower order terms are of comparable magnitude; there is an enormous difference between comparing a sum of $1$ versus $1/p$ and a sum of $1/p$ versus $1/p^2$. \end{rek}


In \S\ref{second_moms} we verify the elliptic curve Bias Conjecture for several one parameter families over $\Q(T)$, and then extend to include higher moments for some families with constant $j(T)$ in \S\ref{sec:constantj}; recent work by Kazalicki and Naskrecki \cite{KN1, KN2} have proved the bias conjecture for several families under standard conjectures; some recent work suggests that the bias conjecture might fail for families with moderate rank (see Appendix \ref{sec:rogerwang}). We turn to Dirichlet $L$-functions in \S\ref{sec:dirichlet}, and see in Theorems \ref{theorem:biasallchar} and \ref{theorem:ltorsionbias} that under GSH the bias is sometimes positive and sometimes negative, similar to the behavior seen in investigating Chebyshev's bias. Using the Petersson formula we show in Theorem \ref{theorem:leveltoinfinitysymr} that there can be a small positive bias for cuspidal newforms in \S\ref{sec:gltwo}, and then conclude in \S\ref{sec:convolution} by investigating how the bias behaves under convolution of families. Theorem \ref{theorem:dirichletconvolved} looks at two families of Dirichlet characters, Theorem \ref{theorem:dirichletcuspgrow} replaces one of those families with a family of $r$\textsuperscript{{\rm th}} symmetric lifts of cuspidal newforms, and Theorem \ref{theorem:convcuspgrow} studies two symmetric lift families. There is tremendous universality in the behavior of the zeros of families of $L$-functions; there are very few main terms that are seen for the $n$-level densities; however, by looking at the lower order terms we can distinguish the behavior in different families. It is here, in the lower order terms, that the arithmetic of the family lives, and the motivation of this paper is to determine, or at least see and leave to others to investigate further, how the arithmetic affects the behavior of the zeros.


\section{Linear one-parameter families of elliptic curves} \label{second_moms}

In this and the next section we amass more evidence for the Bias Conjecture by demonstrating negative bias in additional one-parameter families of elliptic curves. See \cite{MMRW,Mi1,Mi3} for earlier calculations on the subject. The families studied are amenable to direct calculation; thus these are not generic families but ones chosen so that the resulting Legendre sums are tractable.

We collect several standard lemmas for calculating biases in elliptic curve families. Throughout this paper, $\js{\cdot}$ denotes a Legendre symbol, and $\sum_{x(p)}$ denotes a sum over all residue classes modulo $p$. Linear sums and quadratic sums of Legendre symbols can be easily evaluated but cubic and higher typically cannot (which is why it is so hard to work with elliptic curves, as these give rise to cubic Legendre sums). The idea behind all the constructions is to bypass the difficulty of dealing with these cubic and higher Legendre sums by looking at carefully chosen one-parameter families where we can execute the sum over $t$ modulo $p$ exactly \emph{and} be left with sums of a special form that can be evaluated in closed form. We state below the result for linear and quadratic sums; see for example \cite{BEW98, Mi1} for the standard proof.

\begin{lem}\label{lem:sumlegendrelinquad} Let $a, b, c$ be positive integers, and assume $p \nmid a$. Then
\be \label{eq:linearlegendre}
\sum_{x\bmod p} \left(\frac{ax+b}{p}\right) \ = \  0
\ee
and, letting $\Delta = b^2 - 4ac$ (the discriminant in $x$ of $ax^2 + bx + c$),
\be \label{eq:quadraticlegendre}
\sum_{x\bmod p} \left(\frac{ax^2+bx+c}{p}\right) \ = \  \begin{cases} -\js{a} & \mbox{{\rm if} } p \nmid \Delta \\ (p - 1)\js{a} & \mbox{{\rm if} } p \mid \Delta. \end{cases}
\ee
\end{lem}




We first investigate families of the form $\mathcal E_t: y^2 = P(x)t + Q(x)$ where $P$ and $Q$ are fixed polynomials. For convenience, we suppose $p$ is an odd prime. Recalling $a_{\mathcal{E}_t}(p) = -\sum_{x(p)}\legendre{P(x)t + Q(x)}{p}$, we have
\begin{align*}
A_{2,\mathcal E}(p) & \ = \ \sum_{t \bmod p} a_{\mathcal{E}_t}(p)^2 \\
&= \ \sum_{t\bmod p} \sum_{x\bmod p} \sum_{y\bmod p} \legendre{P(x)t + Q(x)}{p} \legendre{P(y)t + Q(y)}{p} \\
&= \ \sum_{x\bmod p} \sum_{y\bmod p} \left[\sum_{t\bmod p} \legendre{P(x)P(y)t^2 + \left(P(x)Q(y) + Q(x)P(y)\right)t + Q(x)Q(y)}{p}\right].\\
\end{align*}
If $P(x)P(y)$ is not zero modulo $p$ we have a quadratic in $t$, with discriminant \bea \Delta(x,y) \ := \ \left(P(x)Q(y) + Q(x)P(y)\right)^2 - 4 P(x)P(y)Q(x)Q(y) \ = \ \left(P(x)Q(y) - Q(x)P(y)\right)^2;\nonumber\\ \eea note $\Delta(x,y) = 0$ if and only if $P(x)Q(y) - Q(x)P(y) = 0$. Applying Lemma \ref{lem:sumlegendrelinquad} to the quadratic sum over $t$ yields
\begin{align}\label{eq:expansionlinearellcurvesum}
A_{2,\mathcal E}(p) & \ = \ p\bigg[\sum_{P(x)\equiv0}\legendre{Q(x)}{p}\bigg]^2\ + \ \sum_{x,y\bmod p} \begin{cases} -\js{P(x)P(y)} & \mbox{{\rm if} } p \nmid \Delta(x,y) \\ (p - 1)\js{P(x)P(y)} & \mbox{{\rm if} } p \mid \Delta(x,y) \end{cases}\bigg\}\nonumber\\
& = \ p\bigg[\sum_{P(x)\equiv0}\legendre{Q(x)}{p}\bigg]^2 \ - \ \bigg[\sum_{x\bmod p}\legendre{P(x)}{p}\bigg]^2 \ + \ p\sum_{\Delta(x,y)\equiv 0}\legendre{P(x)P(y)}{p}.
\end{align}

We use this formula to compute $A_{2,\mathcal E}(p)$ for families whose sum over $\Delta(x,y)\equiv 0$ is tractable.


\begin{proposition} \label{prop:fam1}
The one-parameter family
\begin{align}
\mathcal{E}: y^2 & \ = \ (ax+b)(cx^2+dx+e+T)
\end{align}
with $p>3$ prime, $a,b,c,d,e\in\Z$ and $p\nmid a,c$ has vanishing first moment, hence rank zero, and second moment
\begin{align}
A_{2,\mathcal E}(p) \ = \ \begin{cases}
p^2-\left(2+\legendre{-1}{p}\right)p &{\rm if}\  p\nmid ad-2bc\\
(p^2-p)\left(1+\legendre{-1}{p}\right) &{\rm if}\  p\mid ad-2bc.
\end{cases}
\end{align}
\end{proposition}

\begin{proof} We write $\mathcal E_t: y^2 = P(x)t + Q(x)$ where $P(x)=ax+b$ and $Q(x)=(ax+b)(cx^2+dx+e)$. The first moment is
\begin{align}
A_{1,\mathcal E}(p) & \ = \ \sum_{t \bmod p} a_{\mathcal{E}_T}(p)\nonumber\\
& \ = \ -\sum_{t\bmod p} \sum_{x\bmod p}\legendre{P(x)t + Q(x)}{p} \nonumber\\
& \ = \ -\sum_{x\bmod p}\legendre{ax+b}{p}\sum_{t\bmod p} \legendre{t + (cx^2+dx+e)}{p} = 0
\end{align}
by Lemma \ref{lem:sumlegendrelinquad} applied to the sum over $t\bmod{p}$. Hence $\mathcal E$ has rank zero as it is a rational surface.

Write the family as $\mathcal E_t: y^2 = P(x)t + Q(x)$ for $P(x)=ax+b$ and $Q(x)=P(x)(cx^2+dx+e)$. Then $P(x)\mid Q(x)$ and $\sum_{x(p)}\legendre{P(x)}{p}=0$ so by \eqref{eq:expansionlinearellcurvesum}, we have
\bea A_{2,\mathcal E}(p) = p\sum_{\Delta(x,y)\equiv 0}\legendre{P(x)P(y)}{p}. \eea

Since
\begin{align}
P(y)Q(x) - P(x)Q(y) & \ = \ (ax+b)(ay+b)[(cx^2+dx+e)-(cy^2+dy+e)] \nonumber\\
& \ = \ (ax+b)(ay+b)(x-y)(cx+cy+d),
\end{align}
we deduce $\Delta(x,y) \equiv 0$ if and only if $P(x)\equiv0$, $P(y)\equiv0$, $x\equiv y$, or $x+y\equiv-d/c$. Thus by inclusion-exclusion
\begin{align}
\sum_{\Delta(x,y)\equiv 0} \legendre{P(x)P(y)}{p} & \ = \ \sum_{x+y\equiv-d/c}\legendre{P(x)P(y)}{p} + \sum_{x\equiv y}\legendre{P(x)P(y)}{p}- \sum_{x+y\equiv-d/c \atop x\equiv y}\legendre{P(x)P(y)}{p} \nonumber\\
& \ = \ \sum_{x\;(p)}\legendre{P(x)P(-x-d/c)}{p} + (p-1) - \legendre{P(-d/2c)^2}{p}.
\end{align}
We have $P(x)P(-x-d/c)=-a^2x^2 - a^2d/cx - abd/c + b^2$ so that
\begin{align}
\sum_{x\;(p)}\legendre{P(x)P(-x-d/c)}{p}& \ = \ \legendre{-1}{p}\cdot\begin{cases}
-1 &{\rm if\ } p\nmid ad-2bc\\
p-1&{\rm if\ } p\mid ad-2bc
\end{cases}
\end{align}
since $P(x)P(-x-d/c)$ has discriminant $a^2(ad/c-2b)^2$ and $p\nmid a,c$. Also note that
\begin{align}
\legendre{P(-d/2c)}{p}^2 \ = \ \legendre{-ad/2c+b}{p}^2 \ = \ \begin{cases}
1 &{\rm if\ } p\nmid ad-2bc\\
0 &{\rm if\ } p\mid ad-2bc.
\end{cases}
\end{align}
Plugging into the above gives
\begin{align}
\sum_{\Delta(x,y)\equiv 0} \legendre{P(x)P(y)}{p} \ = \ \begin{cases}
-\legendre{-1}{p}+(p-1)-1 &{\rm if\ } p\nmid ad-2bc\\
(p-1)\legendre{-1}{p}+(p-1)-0 &{\rm if\ } p\mid ad-2bc
\end{cases}
\end{align}
and thus
\begin{align}
A_{2,\mathcal E}(p) \ = \ p\sum_{\Delta(x,y)\equiv 0} \legendre{P(x)P(y)}{p} =
\begin{cases}
p^2-\left(2+\legendre{-1}{p}\right)p &{\rm if\ } p\nmid ad-2bc\\
(p^2-p)\left(1+\legendre{-1}{p}\right) &{\rm if\ } p\mid ad-2bc.
\end{cases}
\end{align}
\end{proof}

If $ad - 2bc$ is not zero, then for all sufficiently large $p$ we have $p\nmid ab-2bc$, and thus by Dirichlet's theorem for primes in arithmetic progression the main term is $p^2$ and half the time the leading lower order term is $-3p$ and half the time it is $-p$.


We compute the first and second moments of three other one-parameter families. The proofs are similar to Proposition \ref{prop:fam1}; see Appendix \ref{sec:linearellipticfamilies} for details.

\begin{proposition} \label{prop:fam2}
The one-parameter family
\begin{align}
\mathcal E: y^2 & \ = \  (ax^2+bx+c)(dx+e+T)
\end{align}
with $p>3$ prime, $a,b,c,d,e\in\Z$ and $p\nmid a,d$ has vanishing first moment, hence rank zero, and second moment given by
\begin{align}
A_{2,\mathcal E}(p) & \ = \ \begin{cases}
p^2-\left(1+\legendre{b^2-4ac}{p}\right)p - 1 &{\rm if\ } p\nmid b^2-4ac\\
p - 1 &{\rm if\ } p\mid b^2-4ac.
\end{cases}
\end{align}
\end{proposition}

\begin{proposition}\label{prop:fam3}
The one-parameteter family
\begin{align}
\mathcal E: y^2 & \ = \  x(ax^2+bx+c+dTx)
\end{align}
with $p>3$ prime, $a,b,c,d\in\Z$ and $p\nmid a,d$ has vanishing first moment, hence rank zero, and second moment given by
\begin{align}
A_{2,\mathcal E}(p) \  = \ \begin{cases}
p^2 - 2 p - 1 &{\rm if\ } p\nmid c\\
p - 1 &{\rm if\ } p\mid c.
\end{cases}
\end{align}
\end{proposition}

\begin{proposition} \label{prop:fam4}
The one-parameteter family
\begin{align}
\mathcal E: y^2 & \ = \  x(ax+b)(cx+d+Tx)
\end{align}
with $p>3$ prime, $a,b,c,d\in\Z$ and $p\nmid a,d$ has vanishing first moment, hence rank zero, and second moment given by
\begin{align}
A_{2,\mathcal E}(p) \  = \ \begin{cases}
p^2-2p - 1 &{\rm if\ } p\nmid b\\
p^2 - p &{\rm if\ } p\mid b.
\end{cases}
\end{align}
\end{proposition}


\section{Elliptic curve families of constant \texorpdfstring{$j(T)-$}{j(T)-}invariant}\label{sec:constantj}

So far, all of the families of elliptic curves that have been investigated for bias have had non-constant $j(T)$-invariant (in $T$). This is motivated by a result of Michel \cite{Mic}, which states that for such families,
\begin{align}
    A_{2, \mathcal{E}}(p)\ = \ p^2 + O \left (p^{3/2} \right ).
\end{align}
In this section, we study families of elliptic curves that have constant $j(T)$-invariant and observe that, for these families, the Bias Conjecture does not seem to apply in a sensible way. Moreover, our methods  allow us to compute moments of \textit{arbitrary} degree. This is notable because, in general, it is extremely difficult to compute moments higher than second of elliptic curve families, due to the complexity of the Legendre sums.

In Section \ref{j1728_moms}, we study the elliptic curve family $\mathcal E^r(T): y^2=x^3-T^rAx$, for any $r \in \N$ and $A \neq 0$. Such families have constant $j(T)$-invariant equal to $1728$. Our main result there is Theorem \ref{thm:j1728_main}.

In Section \ref{j0_moms}, we study the elliptic curve family $\mathcal E^r(T): y^2=x^3+T^r B$, for any $r \in \N$ and $B \neq 0$. Such families have constant $j(T)$-invariant equal to $0$. Our main result there is Theorem \ref{thm:j0_main}.

In Section \ref{sec:jarb}, we compute the $k$\textsuperscript{th} moment of the elliptic curve family $\mathcal{E}(T): y^2 \ = \ x^3+T^2Ax + T^3B$, for any $k$. Note that such families have constant $j(T)$-invariant given in terms of $A$ and $B$. Our main result there is Theorem \ref{thm:jarb}.

Since the moments of elliptic curve families are intimately related to the number of points on each elliptic curve within the family, we begin by detailing known results on counting points of elliptic curves of $j$-invariant $0$ and $1728$. For an elementary survey of results on counting points of elliptic curves of fixed $j$-invariant, see \cite{BKNT}.


\subsection{Counting Points Preliminaries}

\subsubsection{Elliptic Curves of $j$-invariant $0$}
Note that all elliptic curves with $j$-invariant $0$ over a finite field take the form $y^2 = x^3+B$. The following lemma shows that the order of an elliptic curve group over $\F_p$ is determined by the sextic residue class of $B$.

\begin{lemma}[Proposition 2.1, \cite{BKNT}]\label{eq_sextic}
The elliptic curves $E_1: y^2 = x^3+B$ and $E_2: y^2 = x^3+t^6B$, where $t, B \in \F_p^{\times}$, have the same number of points over $\F_p$.
\end{lemma}

Suppose $p\equiv 1\bmod{6}$ is a prime. We can partition $\F_p^{\cross}$ into six equivalence classes via the relation
\[
    k_1 \sim k_2,
    \iff
    \exists k \in \F_p^{\cross} \ \ \ \text{such that} \ \ \  k_1 k_2^{-1}= k^6,
\]
for $k_1, k_2 \in \F_p^{\cross}$.
The second condition above states that $k_1 k_2^{-1}$ is a sixth power in $\F_p^{\cross}$.
These equivalence classes, which all have size $(p-1)/6$, will be referred to as \emph{sextic residue classes}. For a prime $p\not\equiv1\bmod{6}$, we can still partition into sextic residue classes via the above equivalence relation, although there will not be six distinct equivalence classes.

From the above proposition, it follows that if $k_1, k_2 \in \F_p^{\cross}$ are in the same sextic residue class, then the curves $E_1: y^2 = x^3 + k_1$ and $E_2:y^2 = x^3 + k_2$ have the same number of $\F_p$-points.
Hence an elliptic curve $E: y^2=x^3+k$ (i.e., of $j$-invariant $0$) takes on one of at most six different orders for the group of points $E(\F_p)$, where the number of distinct orders depends on the prime $p$. The above discussion of sextic residue classes implies that when $p\equiv 1\bmod{6}$ (or equivalently, $p\equiv 1\bmod{3}$), six distinct orders are realized, whereas when $p\equiv 2\bmod{3}$, less than six distinct orders are realized. Below, we state in Theorem \ref{six_ord} a result of Gauss that gives these orders explicitly.
%
%
Before doing so, it is convenient to introduce the notation $a_E(p)$, for an elliptic curve $E: y^2 = x^3 +ax+b$, defind to be
\begin{equation}\label{eq:aEp}
a_{E}(p) \ := \ p \ - \  \#\{(x,y) \in (\Z/p\Z)^2 :y^2 \equiv x^3+ax + b \;(p)\} \ = \ - \sum_{x \bmod p} \legendre{x^3+ax+b}{p}.
\end{equation}

\begin{theorem}[Gauss, six orders for $j=0$ curves] \label{six_ord}

Let $E$ be the elliptic curve $y^2 = x^3+B$ (i.e., a curve with $j-$invariant $0$) with good reduction mod $p$.
\begin{enumerate}
    \item If $p \equiv 2 \bmod 3$, then $a_{E}(p)=0$.

    \item If $p \equiv 1 \bmod 3$, write $p$ uniquely as $p = a^2+3b^2$, where $a \equiv 2 \bmod 3$ and $b>0$. Then
\begin{align}
a_E(p)\ =\
\begin{cases}
-2a & \ \text{{\rm for $B$ a sextic residue}}    \bmod p\\
2a & \ \text{{\rm for $B$ a cubic, non-quadratic residue}} \bmod p\\
a \pm 3b & \ \text{{\rm for $B$ a quadratic, non-cubic residue}}\bmod p\\
-a \pm 3b & \ \text{{\rm for $B$ a non-quadratic, non-cubic residue}} \bmod p.
\end{cases}
\end{align}
\end{enumerate}
\end{theorem}



We note that as $B$ runs through the values of $\F_p^{\times}$, $a_E(p)$ is equal to each of $\{\pm2a,\pm a \pm 3b\}$ with equal proportion, since the sextic residue classes equipartition $\F_p^{\times}$. The choices of signs in the expressions $a \pm 3b$ and $-a \pm 3b$ can be specified by an analog of the Legendre symbol for a fixed $\pi\in\Z[(1+\sqrt{-3})/2]$ lying over $p$, which will be specified by the choice of $a$ and $b$. See \cite[{Theorem 18.3.4}]{IR} for the statement and proof.



\subsubsection{Elliptic Curves of $j$-invariant $1728$}
All elliptic curves with $j$-invariant $1728$ are of the form $E: y^2= x^3-Ax$. Similar to the $j=0$ case, the following lemma shows that the order of an elliptic curve group over $\F_p$ is determined by the quartic residue class of $A$.

\begin{lemma}[Proposition 3.1, \cite{BKNT}]\label{eq_quartic}
For any $t \in \F_p^{\times}$, $E_1: y^2 = x^3-Ax$ and $E_2: y^2 = x^3-t^4Ax$ have the same number of points over $\F_p$.
\end{lemma}

The following theorem of Gauss gives the explicit dependence of $a_E(p)$ on the quartic residue class of $A$, for $E$ an elliptic curve with $j$-invariant equal to $1728$.

\begin{theorem}[Gauss, four orders for $j=1728$ curves] \label{four_ord}
Let $E: y^2 = x^3-Ax$ (i.e., a $j=1728$ curve) with good reduction mod $p$.
\begin{enumerate}
    \item If $p \equiv 3 \bmod 4$, then $a_{E}(p)=0$.

    \item If $p \equiv 1 \bmod 4$, write $p$ as $p = a^2+b^2$, where $b$ is even and $a+b \equiv 1 \bmod 4$. Then
\begin{align}
    a_{E}(p)\ =\
    \begin{cases}
    2a &\text{{\rm for $A$ a quartic residue}}\bmod p\\
    -2a &\text{{\rm for $A$ a quadratic, non-quartic residue}}\bmod p\\
    \pm 2b &\text{{\rm for $A$  a  non-quadratic residue}}\bmod p.
    \end{cases}
\end{align}
\end{enumerate}
\end{theorem}
For proofs, see \cite[{Theorem 18.4.5}]{IR} or \cite{Wash}. Again, the sign in the expression $\pm 2b$ may be specified via a quartic residue symbol in the ring $\Z[i]$.
We note that as $A$ runs through the values of $\F_p^{\times}$, $a_E(p)$ is equal to each of the elements of the set $\{\pm2a,\pm 2b\}$ with equal proportion--- this is because the cosets of the image of the fourth-power map evenly partition $\F_p^{\times}$.


\subsection{\texorpdfstring{$k$\textsuperscript{th}}{kth} Moments of the family \texorpdfstring{$\mathcal{E}^r(T): y^2=x^3-T^rAx$}{Er(T): y2=x3-TrAx}} \label{j1728_moms}

We now determine the $k$\textsuperscript{th} moment at $p$ of the elliptic curve family $\mathcal E^r(T): y^2=x^3-T^rAx$ for any $k, r \in \N$ and $A \neq 0$.
These families all have constant $j$-invariant $j(T) = 1728$.
Note that it suffices to consider the equivalence class of $r \bmod 4$, since $y^2 = x^3 -T^rAx$ and $y^2 = x^3 -T^{r+4}Ax$ have the same number of $\F_p$ points by Lemma \ref{eq_quartic}. Thus, for $r \in \{0,1,2,3\}$, we may define the quantity
\[
    A_{k, \mathcal{E}^{[r]_4}}(p) := A_{k, \mathcal{E}^{r'}}(p),
\]
which gives the $k^{\mathrm{th}}$ moment at $p$ for any $r' \in \N$ such that $r' \equiv r \bmod 4$.

Our main result is the following.

\begin{theorem} \label{thm:j1728_main}
For $r \in \N$ and $A \neq 0$, consider the elliptic curve family $\mathcal{E}^r(T): y^2 = x^3 -T^r Ax$.
\begin{enumerate}
    \item If $p \equiv 3 \bmod 4$, then for all $k, r \in   \N$,  we have
    \begin{align}
        A_{k, \mathcal{E}^r}(p)\ =\ 0.
    \end{align}

    \item For $p \equiv 1 \bmod 4$, we have the following results:
    \begin{align}
        &A_{k, \mathcal E^{[0]_4}}(p)
            \ =\
            \begin{cases}
                (p-1)(2a)^k
                &\text{{\rm for} $A$  {\rm a quartic residue}}\bmod p\\
                (p-1)(-2a)^k
                &\text{{\rm for} $A$  {\rm a quadratic, non-quartic residue}}\bmod p\\
                 (p-1)(\pm2b)^k
                 &\text{{\rm for} $A$  {\rm a non-quadratic residue}}\bmod p,
            \end{cases}
            \label{eq:1728_0}
            \\
        &A_{k, \mathcal E^{[1]_4}}(p)
            \ = \
            A_{k,\mathcal{E}^{[3]_4}}(p)
            \ = \
            \begin{cases}
            (p-1)2^{k-1}(a^k+b^k)& \text{{\rm for $k$ even}}\\
            0 & \text{{\rm for $k$ odd}},
            \end{cases}
            \label{eq:1728_1,3}
            \\
        &\text{ and finally,} \nonumber
           \\
        &A_{k, \mathcal E^{[2]_4}}(p)
            \ =\
            \begin{cases}
                (p-1)(2a)^k& \text{{\rm for $k$ even}}, \text{{\rm $A$  a quadratic residue}}\bmod p\\
                (p-1)(2b)^k& \text{{\rm for $k$ even}}, \text{{\rm $A$ a non-quadratic residue}}\bmod p\\
                0 & \text{{\rm for $k$ odd}},
                \label{eq:1728_2}
        \end{cases}
    \end{align}
\end{enumerate}
where $a$, $b$, and the plus/minuses in \eqref{eq:1728_0} are given by Theorem \ref{four_ord}.
\end{theorem}

\begin{remark}
We can trivially extend these results to any $\mathcal{E}(T): y^2 = x^3 - (cT+d)^r Ax$, since $T\mapsto cT+d$ is a bijection of $\F_p$ for $p\nmid c$.
\end{remark}

\begin{proof}[Proof of Theorem \ref{thm:j1728_main}]
   If $p \equiv 3 \bmod 4$, then by Theorem \ref{four_ord}, every  elliptic curve $E$ with $j = 1728$ has $a_E(p) = 0$; thus,
it follows trivially that $A_{k,\mathcal{E}^r} = 0$ for all $k, r \in \N$.

It remains to consider the $p \equiv 1 \bmod 4$ case. For organizational clarity, we prove equation \eqref{eq:1728_0} in Section \ref{j1728_triv},
equation \eqref{eq:1728_1,3} in Sections \ref{sec:1728_1proof} and \ref{sec:1728_3proof}, and equation \eqref{eq:1728_2} in Section \ref{sec:1728_2proof}.
\end{proof}

Per the above, we will always consider $p \equiv 1 \bmod 4$ for the rest of the calculations in this subsection.
It will often be convenient to define $a_p(x^3+ax +b) := a_{E}(p)$, for $E:y^2=x^3+ax+b$ and $a_E(p)$ defined as in \eqref{eq:aEp}. Then, we can write
\begin{align}\label{four_kmom}
    A_{k, \mathcal E^r}(p)\ =\ \sum_{T=1}^{p-1} a_p(x^3-T^rAx)^k \  = \  \sum_{T \in \F_p^{\times}} a_p(x^3-T^rAx)^k.
\end{align}

\subsubsection{Proof of equation \eqref{eq:1728_0}: $r\equiv 0 \bmod 4$} \label{j1728_triv}
In the $r\equiv 0 \bmod 4$ case, we need only consider the constant family $\mathcal{E}^0(T):y^2=x^3-Ax$. Then $A_{k, \mathcal E^0}(p)\ =\  (p-1)a_{\mathcal E}(p)^k$, and so equation \eqref{eq:1728_0} follows from Theorem \ref{four_ord}.


\subsubsection{Proof of equation \eqref{eq:1728_1,3}: $r \equiv 1 \bmod 4$} \label{sec:1728_1proof}
We examine $\mathcal{E}^1(T): y^2 = x^3-TAx$. Note that $TA$ runs through the elements of $\F_p^{\times}$ as $T$ does, and since $(p-1)/4 \in \Z$, $a_p(x^3-TAx)$ takes on each of the values given in Theorem \ref{four_ord} a total of $(p-1)/4$ times.  Substituting into \eqref{four_kmom} gives
\begin{align}
    A_{k, \mathcal E^1}(p) &\ = \  \frac{p-1}{4} \left((2a)^k + (-2a)^k +(2b)^k +(-2b)^k\right) \nonumber \\
    &\ = \  \begin{cases}
    (p-1)2^{k-1}(a^k+b^k)& \text{for $k$ even}\\
    0 & \text{for $k$ odd,}
    \end{cases}
\end{align}
as desired.


\subsubsection{Proof of equation \eqref{eq:1728_1,3}: $r \equiv 3 \bmod 4$}
\label{sec:1728_3proof}
We now examine $\mathcal{E}^3(T): y^2=x^3-T^3Ax$. Select a generator $g$ of $\F_p^{\cross}$. Since $A \neq 0$, we have  $A\equiv g^m\bmod{p}$ for some unique $m \in \{1,\dots, p-1\}$. The reduction of the set  $\{T^3A:1\leq T\leq p-1\}$ modulo $p$ may then be written as $\{g^{3t+m}:1\leq t\leq p-1\}$.

By Theorem \ref{four_ord}, we know $a_p(x^3-T^3 Ax)$ depends only on the ``quartic coset'' of $T^3A$ modulo $p$; in other words, there is a unique $i \in \{0,1,2,3\}$ for which $T^3 A \in g^i (\F_p^\times)^4$ and $a_p(x^3-T^3 Ax)$ depends only on $i$. Moreover, $g^{3t+m}$ and $g^{3t'+m}$ lie in the same quartic coset if and only if $t \equiv t' \bmod{4}$.
Since $p \equiv 1 \bmod{4}$, the congruence classes of $t$ modulo $4$ partition the set $\{t:1\leq t\leq p-1\}$ equally into four sets, each of size $(p-1)/4$. We conclude that the elements of the set $\{T^3A: 1 \leq T\leq p-1\}$ are uniformly distributed among all four quartic cosets. We may then compute as in the $r \equiv 1\pmod{4}$ case to find that
\begin{align}
A_{k, \mathcal E^3}(p)\ =\ A_{k, \mathcal E^1}(p).
\end{align}
This shows equation \eqref{eq:1728_1,3}.

\subsubsection{Proof of equation \eqref{eq:1728_2}: $r \equiv 2 \bmod 4$}
\label{sec:1728_2proof}

We examine $\mathcal{E}^2(T): y^2 = x^3-T^2Ax$. If $A$ is a quadratic residue, then $T^2A$ is a quadratic residue for all $T\in\F_p^\times$. Moreover, writing $A \equiv a^2 \bmod p$, we have $T^2A$ is a quartic residue whenever $\legendre{T}{p} = \legendre {a}{p}$, which occurs half of the time. Thus, from \eqref{four_kmom}  and Theorem \ref{four_ord}, we have
\begin{align}
    \label{specificMomentone}
    A_{k, \mathcal E^2}(p) \ =\ \frac{p-1}{2} \left((2a)^k + (-2a)^k\right)\ =\ \begin{cases}
    (p-1)(2a)^k& \text{for $k$  even}\\
    0 & \text{for $k$  odd.}
    \end{cases}
\end{align}
On the other hand, if $A$ is a quadratic non-residue, then $T^2A$ is a quadratic non-residue for all $T\in\F_p^\times$, with half of these values in each non-quadratic coset of the image of the fourth-power map $\phi: x \mapsto x^4$ (fix a generator $g\in \F_p^{\times}$--- the non-quadratic cosets are $g \text{Im}(\phi)$ and $g^3 \text{Im}(\phi)$; the quadratic cosets are $\text{Im}(\phi)$ and $g^2 \text{Im}(\phi)$). What this means is that $a_p$ takes on each value in the ``non-quadratic'' case of Theorem \ref{four_ord} half of the time as $T$ runs through $\F_p^{\times}$. Thus,
\begin{align}
A_{k, \mathcal E^2}(p)\ = \ \frac{p-1}{2}\left((2b)^k+(-2b)^k\right)\ =\ \begin{cases}
    (p-1)(2b)^k& \text{for $k$  even}\\
    0 & \text{for $k$  odd.}
    \end{cases}
\end{align}
Hence, we have shown equation \eqref{eq:1728_2}.
This concludes the proof of Theorem \ref{thm:j1728_main}


\subsection{\texorpdfstring{$k$\textsuperscript{th}}{kth} Moments of the family \texorpdfstring{$\mathcal{E}^r(T): y^2=x^3+T^rB$}{Er: y2=x3-TrB}} \label{j0_moms}

We now analyze  families of elliptic curves of the form $\mathcal E^r(T): y^2=x^3+T^rB$, for all $r \in \N$.  These families have constant $j$-invariant $j(T) = 0$. Similar to the $j=1728$ case, note that it suffices to consider $r$ modulo $6$ by Lemma \ref{eq_sextic}.
Thus, for $r \in \{0,1,\dots, 5\}$, we may consider the quantity $A_{k, \mathcal{E}^{[r]_6}}(p) := A_{k, \mathcal{E}^{r'}}(p)$,  where $r' \equiv r \bmod 6$.
Following equation \eqref{four_kmom}, we may write
\begin{equation}\label{six_kmom}
   A_{k,\mathcal{E}^{[r]_6}}(p)
   \ = \
   A_{k,\mathcal{E}^{r}}(p)
   \ = \
   \sum_{T=0}^{p-1} a_p(x^3+T^rB)^k \ =\ \sum_{T \in \F_p^{\times}} a_p(x^3+T^rB)^k.
\end{equation}

Our main result is the following.

\begin{theorem} \label{thm:j0_main}
For $r \in \N$, consider the elliptic curve family $\mathcal{E}^r(T): y^2 = x^3 +T^r B$. If $p \equiv 2 \bmod 3$, then for all $k, r \in \N$,
\begin{align}
    A_{k, \mathcal{E}^r}(p) = 0.
\end{align}
If $p \equiv 1 \bmod 3$, we have the following results:
\begin{align}
    A_{k, \mathcal E^{[0]_6}}(p)
&\ =\
\begin{cases}
(p-1)(-2a)^k & \text{{\rm for $B$ a sextic residue}}\bmod p\\
(p-1)(2a)^k & \text{{\rm for $B$ a cubic, non-quadratic residue}}\bmod p\\
(p-1)(a \pm 3b)^k & \text{{\rm for $B$ a quadratic, non-cubic  residue}}\bmod p\\
(p-1)(- a \pm 3b)^k & \text{{\rm for $B$ a non-quadratic, non-cubic residue}} \bmod p,
\end{cases}
    \label{eq:j0_0}
    \\
    A_{k,\mathcal{E}^{[1]_6}}(p)
    &\ =\
    A_{k,\mathcal{E}^{[5]_6}}(p)
    \ =\
    \begin{cases}
        \frac{p-1}{3}\left ( (2a)^k + (a-3b)^k + (a+3b)^k \right )
        &\text{{\rm for $k$ even}}\\
        0 & \text{{\rm for $k$ odd}},
    \end{cases}
    \label{eq:j0_15}
    \\
    A_{k,\mathcal{E}^{[2]_6}}(p)
    &\ = \
    A_{k,\mathcal{E}^{[4]_6}}(p) \nonumber\\
    &
    \ = \
    \begin{cases}
        \frac{p-1}{3} \left ((-2a)^k +(a-3b)^k +(a+3b)^k \right )
        &\text{{\rm for $B$ a quadratic residue}}
        \\
        \frac{p-1}{3} \left ((2a)^k +(-a-3b)^k +(-a+3b)^k \right )
        &\text{{\rm otherwise}}
    \end{cases}
    \label{eq:j0_24}
    \\
    \text{and finally,}&
    \nonumber
    \\
    A_{k,\mathcal{E}^{[3]_6}}(p)
    & \ = \
    \begin{cases}
        (p-1)(2a)^k
        &\text{{\rm for $B$ a cubic residue, $k$ is even}}
        \\
        0 &\text{{\rm for $B$ a cubic residue, $k$ is odd}}
        \\
        \frac{p-1}{2} \left( (a\pm 3b)^k + (-a \mp 3b)^k \right)
        &\text{for $B$ a cubic non-residue,}
    \end{cases}
    \label{eq:j0_3}
\end{align}
where $a$, $b$, and the plus/minuses in \eqref{eq:j0_0} and \eqref{eq:j0_3} are given by Theorem \ref{six_ord}.

\end{theorem}

\begin{remark}
Just as before, it is trivial to extend these results to any $\mathcal{E}(T): y^2 = x^3 + (cT+d)^r B$.
\end{remark}

\begin{proof}[Proof of Theorem \ref{thm:j0_main}]
If $p \equiv 2 \bmod 3$, then by Theorem \ref{six_ord},
every elliptic curve $E$ with $j=0$ has $a_E(p) = 0$; thus, it follows trivially that $A_{k,\mathcal{E}^r}(p) = 0$ for all $k, r \in \N$.

For the $p \equiv 1 \bmod 3$ case, similar to the organization in Section \ref{j1728_moms},
we prove equation \eqref{eq:j0_0} in Section \ref{sec:j0_0}, equation \eqref{eq:j0_15} in Section \ref{sec:j0_1}, equation \eqref{eq:j0_24} in Section \ref{sec:j0_2}, and equation \eqref{eq:j0_3} in Section \ref{sec:j0_3}.

\end{proof}

\subsubsection{Proof of equation \eqref{eq:j0_0}: $r \equiv 0 \bmod 6$}
\label{sec:j0_0}
In the $r \equiv 0 \bmod 6$ case, we need only consider the constant family $\mathcal{E}^0(T): y^2 = x^3 +B$. Then $A_{k, \mathcal E^0}(p) = (p-1)a_{p}(x^3+B)^k$, and so
Theorem \ref{six_ord} gives \eqref{eq:j0_0}.


\subsubsection{Proof of equation \eqref{eq:j0_15}: $r \equiv 1, 5 \bmod 6$}
\label{sec:j0_1}
We examine $\mathcal{E}^1(T): y^2 = x^3+TB$.
Since $TB$ runs through all elements of
$\F_p^{\times}$ and since
$\frac{p-1}{6} \in \Z$, equation \eqref{six_kmom} yields
\begin{align}
    A_{k,\mathcal{E}^1}(p) &\ = \  \frac{p-1}{6} \left ((2a)^k + (-2a)^k +(a-3b)^k +(a+3b)^k+(-a+3b)^k +(-a-3b)^k \right )\nonumber\\
    &\ = \  \begin{cases}
    \frac{p-1}{3}\left ( (2a)^k + (a-3b)^k + (a+3b)^k \right )
    &\text{for $k$ even}\\
    0 & \text{for $k$ odd.}
    \end{cases}
\end{align}
We note that the case $r\equiv 5 \bmod 6$ reduces to that of  $r\equiv 1 \bmod 6$ via the isomorphism $T\mapsto T^{-1}$ on $\F_p^{\times}$.
This shows \eqref{eq:j0_15}.

\subsubsection{Proof of equation \eqref{eq:j0_24}: $r \equiv 2, 4 \bmod 6$}
\label{sec:j0_2}
We examine $\mathcal{E}^2(T): y^2 = x^3+T^2B$. Suppose $B$ is a quadratic residue. Then $T^2B$ is always a quadratic residue, runs through each of the quadratic residues of
$\F_p^{\times}$ twice, and is a sextic residue one-third of the time. To see this, fix a generator $g \in \F_p^{\times}$, denote the cosets of the image of the sixth power map by $\{[g^{6c}]$, $[g^{6c+1}]$, $[g^{6c+2}]$, $[g^{6c+3}]$, $[g^{6c+4}]$, $[g^{6c+5}]\}$. Squaring gives $\{[g^{6c}]$, $[g^{6c+2}]$, $[g^{6c+4}]\}$, the three quadratic residue classes. Thus, from Theorem \ref{six_ord}, we see that the values of $a_p$ are split uniformly among $-2a, a-3b$, and $a+3b$.
Indeed, $\frac{p-1}{3}$ is an integer, and so we have
\begin{align}
    A_{k,\mathcal{E}^2}(p)\ =\ \frac{p-1}{3} \left ((-2a)^k +(a-3b)^k +(a+3b)^k \right)
\end{align}
for $B$ a quadratic residue.

On the other hand, suppose $B$ is a quadratic non-residue. Then $T^2B$ is never a quadratic residue, while it is a cubic residue for one-third of the values of $T$. By the same argument as above, we have
\begin{align}
    A_{k,\mathcal{E}^2}(p)\ =\ \frac{p-1}{3} \left ((2a)^k +(-a-3b)^k +(-a+3b)^k \right).
\end{align}
We note that the case $r\equiv 4 \bmod 6$ reduces to that of  $r\equiv 2 \bmod 6$ via the isomorphism $T\mapsto T^{-1}$ on $\F_p^{\times}$.
This shows equation \eqref{eq:j0_24}.

\subsubsection{Proof of equation \eqref{eq:j0_3}: $r \equiv 3 \bmod 6$}
\label{sec:j0_3}
We examine $\mathcal{E}^3(T): y^2 = x^3+T^3B$. If $B$ is a cubic residue, then $T^3B$ is either a sextic or cubic residue, with each case occurring with equal proportion over the values of $T$. Thus,
    \begin{align}
        A_{k,\mathcal{E}^3}(p)\ =\ \frac{p-1}{2} \left( (-2a)^k + (2a)^k \right)\ =\ \begin{cases}
    (p-1)(2a)^k& k\text{ is even}\\
    0 & k\text{ is odd.}
    \end{cases}
\end{align}
On the other hand, if $B$ is a cubic non-residue then
    \begin{align}
        A_{k,\mathcal{E}^3}(p)\ =\ \frac{p-1}{2} \left( (a\pm 3b)^k + (-a \mp 3b)^k \right), \label{bad_T3} 
    \end{align}
where the plus/minus signs are determined by Theorem \ref{six_ord}. This shows \eqref{eq:j0_3}, thereby concluding the proof of Theorem \ref{thm:j0_main}.



\subsection{Computing the \texorpdfstring{$k$\textsuperscript{th}}{kth} moment for a family of any constant \texorpdfstring{$j(T)-$}{j(T)-}invariant}
\label{sec:jarb}

For any constant $j(T)$, the $k$\textsuperscript{th} moment of the elliptic curve family $\mathcal{E}(T): y^2 \ = \ x^3+T^2Ax + T^3B$ can be computed for any $k$, where $j(T) \ =\ 1728\cdot 4A^3/(4A^3+27B^2)$.
\begin{theorem}[$j(T) \neq 0,1728$] \label{thm:jarb}
Consider the elliptic curve family
$
\mathcal{E}(T): y^2\ =\ x^3 + T^2 Ax + T^3 B
$.
We have
\begin{align}
A_{k,\mathcal{E}(T)}(p)\ =\
    \begin{cases}
        (p-1)a_{\mathcal{E}(1)}(p)^k &\text{{\rm for $k$ even}}\\
        0 &\text{{\rm for $k$ odd}}.
    \end{cases}
\end{align}

\end{theorem}

The proof follows from the following lemma, which is well-known.
\begin{lemma} \label{twist_sum}
Fix $d \in \Z$ such that $\legendre{d}{p} = -1$, and consider the elliptic curves $E: y^2 = x^3 + Ax+B$ and $\tilde{E}: y^2 = x^3 +d^2Ax+d^3B$. Then
\[
a_{\tilde{E}}(p) = -a_E(p).
\]
\end{lemma}
\begin{proof}[Proof of Lemma \ref{twist_sum}]
Note that the curve
\begin{align}
dy^2 = x^3 + Ax+B \label{twistE}
\end{align}
has Weierstrass form $y^2 = x^3 + d^2Ax+d^3B$. It follows that the right hand side of (\ref{twistE}) is a quadratic residue if and only if $x^3+d^2Ax+d^3B$ is a non-residue.
\end{proof}

\begin{proof}[Proof of Theorem \ref{thm:jarb}]
Note that when $T \neq 0$, half of the $T$ are quadratic residues and the other half are quadratic non-residues. Since the Legendre symbol is multiplicative, it follows from Lemma \ref{twist_sum} that for half of the curves $E$ in the family, we have $a_E(p) = a_{\mathcal{E}(1)}(p)$, and for the other half, we have $a_E(p) = -a_{\mathcal{E}(1)}(p)$. This immediately yields the desired result.
\end{proof}




\section{\texorpdfstring{${\rm GL}(1)$}{GL(1)} Families (Dirichlet Characters)}\label{sec:dirichlet}

We investigate two families. For the first, we study $\mathcal{D}_q$, the family of nontrivial Dirichlet characters of prime level $q$. For the second, we consider the sub-family of $\mathcal{D}_q$ of characters $\chi$ with prime torsion $\ell$ (thus, $\chi^\ell$ is the principal character). We take $q$ and $\ell$ to be distinct odd primes such that $q \equiv 1 \bmod \ell$, and let $\mathcal{D}_{q, \ell} \subseteq \mathcal{D}_q$ be those $\ell$-torsion characters; note it is not interesting to take $\ell = 2$, as the second moment summand would then be 1 at all primes $p$ relatively prime to the level $q$.

\subsection{Preliminaries: Primes in Arithmetic Progression}

The bias in these families is related to the bias in primes in arithmetic progressions, specifically to the distribution of primes congruent to $1$ or $-1$ modulo a fixed prime $q$. We record some useful facts, taken from \cite{RubSa}.

First, some notation. Let $\pi(X,q,a)$ denote the number of primes at most $X$ which are congruent to $a$ modulo $q$, and set \be\label{eq:defnExqa} E(X,q,a) \ := \ \left(\varphi(q) \pi(X,q,a) - \pi(X)\right) \frac{\log X}{\sqrt{X}}. \ee  By Dirichlet's theorem on primes in arithmetic progression we know that to first order, if $a$ and $q$ are relatively prime, that $\pi(X,q,a)$ and $\pi(X)/\varphi(q)$ are of the same size, and thus the difference $\varphi(q) \pi(X,q,a) - \pi(X)$ should be significantly smaller than $X$; we expect it to be of size $\sqrt{X}/\log X$, hence the normalization factor. Set  \be c(q,a)\ =\ -1 \ + \ \sum_{b^2 \equiv a (q)} 1 \ee and
\begin{align}
\psi(X, \chi) \ = \ \sum_{n < X} \ \Lambda(n) \chi(n)
\end{align}
with $\Lambda(n)$ the classical von-Mangoldt function. We have (see Lemma 2.1 in \cite{RubSa}) that \be E(X,q,a) \ = \ -c(q,a) + \sum_{\chi\neq \chi_0} \overline{\chi}(a) \frac{\psi(X,\chi)}{\sqrt{X}} + O\left(\frac1{\log X}\right).\ee Unfortunately it is difficult to evaluate the sum over characters, though we can express it as a sum over zeros of the associated $L$-functions and then attack its value by assuming the Generalized Riemann Hypothesis (GRH) and the Grand Simplicity Hypothesis (GSH, which states that the imaginary parts of the critical zeros of Dirichlet $L$-functions are linearly independent over the rationals).

\subsection{Characters of Prime Level}

The quantity below is defined to mirror the elliptic curve case (see \eqref{eq:ellcurvesecondmomentfullsum}). We divide by the cardinality of the family $\mathcal{D}_q$, which is $\varphi(q)-1 = q-2$ (remember we are excluding the trivial character).

\begin{definition}\label{definition:DirichletAverageMoment}
The average second moment of the family $\mathcal{D}_q$ is the sum
\begin{align}
M_2(\mathcal{D}_q, X) \ = \ \frac1{\pi(X)} \sum_{p \le X} \frac1{q-2} \sum_{\chi \in \mathcal{D}_q} \ \chi^2(p). \label{eq:secondlocalmmtdirichlet}
\end{align}
\end{definition}

\begin{theorem} \label{theorem:biasallchar}
Assuming GRH, the average second moment $M_2(\mathcal{D}_q, X)$ of the family $\mathcal{D}_q$ converges to $\frac1{q-2}$ as $X \rightarrow \infty$. The lower order error term is $\frac{\sqrt{X}}{(q-2)\pi(X)\log X}\left(E(X,q,1) + E(X,q,-1)\right)$. Additionally assuming GSH and $q$ sufficiently large, the error term as a function of $X$ is sometimes positive and sometimes negative (and on a logarithmic scale each happens a positive percentage of the time).
\end{theorem}

\begin{proof}
When evaluating the sum of Definition \ref{definition:DirichletAverageMoment}, we may assume $p \neq q$ as otherwise the character sum is trivially 0. By $\overline{\chi}$ we mean the inverse of $\chi \bmod q$, and $p^{-1}$ is the inverse of $p$ in $\F_q^\ast$. Rewriting \eqref{eq:secondlocalmmtdirichlet} as
\begin{align}
M_2(\mathcal{D}_q, X) \ = \  \frac1{\pi(X)} \sum_{p \le X}\frac1{q-2} \sum_{\chi \in \mathcal{D}_q} \ \chi(p) \ \overline{\chi(p^{-1})},
\end{align}
we deduce from the Schur orthogonality relations (for sums of Dirichlet characters) that
\begin{align}
\sum_{\chi \in \mathcal{D}_q} \chi(p)^2 \ = \ -1 + \left\{ \begin{array}{ll}
         q - 1  & \mbox{if $p \equiv \pm 1 (q)$} \\
	0 & \mbox{if $p \not\equiv \pm 1 (q)$}. \end{array} \right. \label{eq:schurorthog}
\end{align}
Thus
\begin{align}
M_2(\mathcal{D}_q, X) & \ = \  \frac1{\pi(X)} \frac1{q-2} \left[ \sum_{p \le X \atop p \equiv \pm 1 (q)} \ (q-1) \ - \ \sum_{p \le X} \ 1\right] \nonumber\\
&\ = \  \frac1{q-2} \frac{\varphi(q) \left( \pi(X, q, 1) + \pi(X, q, -1) \right) \ - \ \pi(X)}{\pi(X)} \nonumber\\
&\ = \  \frac1{q-2} \frac{\pi(X) + \left(\varphi(q) \pi(X,q,1) - \pi(X)\right) + \left(\varphi(q) \pi(X,q,-1) - \pi(X)\right)}{\pi(X)} \nonumber\\
& \ = \  \frac1{q-2} + \frac{\sqrt{X}}{(q-2)\pi(X)\log X}\left[E(X,q,1) + E(X,q,-1)\right]. \label{eq:pntterms}
\end{align}

Rubinstein and Sarnak \cite{RubSa} prove that on a logarithmic scale, all possible orderings of the number of primes in distinct residue classes occur a positive percentage of the time. More precisely, if $a_1, \dots, a_r$ are distinct residues relatively prime to $q$, set \be P_{q;a_1,\dots,a_r} \ := \{X\in\R : \pi(X,q,a_1) \ > \ \pi(X,q,a_2) \ > \ \cdots \ > \pi(X,q,a_r)\}. \ee The logarithmic density of a set $P$, denoted $\updelta(P)$, exists (and is the common limit) if the following two limits exist and are equal: \be \overline{\updelta}(P) \ := \ \limsup_{X\to\infty} \frac1{\log X} \int_{t \in P \cap [2,X]} \frac{dt}{t}, \ \ \ \underline{\updelta}(P) \ := \ \liminf_{X\to\infty} \frac1{\log X} \int_{t \in P \cap [2,X]} \frac{dt}{t}. \ee

Theorem 1.5 in \cite{RubSa} states that for each $r\ge2$, assuming GRH and GSH, we have \be \max_{a_1,\dots,a_r\,(q)} \left|\updelta(P_{q;a_1,\dots,a_r}) - \frac1{r!}\right| \ \to \ 0 \qquad (\text{as }q\to\infty). \ee Thus for $q$ sufficiently large, the secondary terms $E(X,q,\pm1)$ in \eqref{eq:pntterms} are each greater than zero (i.e. positive bias) for a positive percentage of $X$ on a logarithmic scale. Similarly $E(X,q,\pm1)<0$ (i.e. negative bias) for a positive percentage of $X$.
\end{proof}

\subsection{Characters of Prime Level and Prime Torsion}

We now fix an odd prime $\ell$ and consider the family $\mathcal{D}_{q, \ell}$ of non-trivial characters of order $\ell$ with $q$ prime and $q \equiv 1 \bmod \ell$, which implies via the structure theorem for finite abelian groups that the families $\mathcal{D}_{q, \ell}$ are nonempty.

\begin{definition}
The average second moment of the family $\mathcal{D}_{q, \ell}$ is
\begin{align}
M_2(\mathcal{D}_{q,\ell}, X) \ = \ \frac1{\pi(X)} \sum_{p \le X} \frac{1}{|\mathcal{D}_{q,\ell}|} \sum_{\chi \in \mathcal{D}_{q,\ell}} \chi^2(p).
\end{align}
\end{definition}

In contrast to the previous family of all characters of level $q$, the restriction to characters with prime torsion $\ell$ gives us a family with zero main term.

\begin{theorem}\label{theorem:ltorsionbias} Fix an odd prime $\ell$. Then the family $\mathcal{D}_{q, \ell}$ has zero main term in its average second moment. Moreover, under GRH and GSH, the bias is positive and negative a positive percentage of the time.
\end{theorem}

To prove Theorem \ref{theorem:ltorsionbias}, we recall some standard properties of the family $\mathcal{D}_{q, \ell}$ (cf. \cite{IK}).

\begin{lemma} \label{lemma:permuteCharacters} If $\ell$ is an odd prime and $r$ is relatively prime to $\ell$, then the map $\chi \mapsto \chi^r$ is an automorphism on $\mathcal{D}_{q, \ell}$.\end{lemma}


\begin{lemma}\label{lemma:countresidues}
Let $\F_q^\ast(\ell) \subseteq \F_q^\ast$ be the $\ell$\textsuperscript{{\rm th}} power residues modulo $q$. Then $\# \F_q^\ast(\ell) = (q-1)/\ell$, which implies $\#\{a \in \F_q^\ast: a \not\in \F_q^\ast(\ell)\} = (q-1) (\ell-1)/\ell$. \end{lemma}


\begin{proof}[Proof of Theorem \ref{theorem:ltorsionbias}] We must compute
\begin{align}\label{eq:m2mathcaldqellp}
M_2(\mathcal{D}_{q, \ell}, X) \ = \ \frac1{\pi(X)} \sum_{p \le X} \frac1{|\mathcal{D}_{q,\ell}|} \sum_{\chi \in \mathcal{D}_{q, \ell}} \chi^2(p).
\end{align}
For $r \in \{1, \dots, \ell-1\}$, we know by Lemma \ref{lemma:permuteCharacters} that the map $\chi \mapsto \chi^r$ is an automorphism of $\mathcal{D}_{q,\ell}$. This gives
\begin{align}
\sum_{\chi \in \mathcal{D}_{q, \ell}} \chi(p)^2 \ = \ \sum_{\chi \in \mathcal{D}_{q, \ell}} \chi(p) \ = \ \frac{1}{\ell - 1} \ \sum_{\chi \in \mathcal{D}_{q, \ell}} \ \left( \chi(p) \ + \ \chi^2(p) \ + \ \cdots \ + \ \chi^{\ell - 1}(p) \right).
\end{align}

For any $\chi \in \mathcal{D}_{q,\ell}$, we have that $p \not \in \F_q^\ast(\ell)$ if and only if $\chi(p) \neq 1$. For such $p$ we have \be 1 + \chi(p) + \chi^2(p) + \cdots + \chi^{\ell-1}(p) \ = \ \frac{\chi^{\ell}(p) - 1}{\chi(p)-1} \ = \ 0,\ee so we find that
\begin{align}
\sum_{\chi \in \mathcal{D}_{q, \ell}} \chi(p)^2 \ = \ \frac{1}{\ell - 1} \ \sum_{\chi \in \mathcal{D}_{q, \ell}} (-1) \ = \ - \frac{|\mathcal{D}_{q, \ell}|}{\ell - 1}
\end{align}
if $p \not \in \F_q^\ast(\ell) $.

If $p = a^{\ell}$ for some $a \in \F_q^\ast$, then trivially $\sum_\chi \chi^2(p) = |\mathcal{D}_{q, \ell}|$ as each $\chi^2(p)$ equals 1. Thus
\begin{align}
M_2(\mathcal{D}_{q, \ell}, X) \ = \ \frac1{\pi(X)} \left( \sum_{a \in \F_q^\ast(\ell)} \ \pi(X, q, a) \ - \ \frac{1}{\ell - 1} \cdot \sum_{a \not\in \F_q^\ast(\ell)} \ \pi(X, q, a)\right). \label{eq:rsexpforsecondmoment}
\end{align}

By Lemma \ref{lemma:countresidues} the number of terms in the sums over $a$ are $(q-1)/\ell$ and $(q-1)(\ell-1)/\ell$ respectively; then by the Prime Number Theorem in arithmetic progressions the main terms cancel above. As before, assuming GRH and GSH, the bias is positive and negative a positive percentage of the time. \end{proof}

\section{Holomorphic Newforms on \texorpdfstring{$GL(2)/\Q$}{GL(2)/Q} and their Symmetric Lifts}\label{sec:gltwo}

For a fixed level $q$ and even integer $k$, let $\mathcal{C}_{k,q}(\chi_0)$ denote the space of cusp forms of weight $k$, level $q$, and principal central character. We assume $q$ is square-free throughout, and our main result will be concerned only with prime $q$ (it should be possible to handle square-free $q$ with more care). Note $\mathcal{C}_{k,q}(\chi_0)$ has a subspace which is an eigenspace for all $p^\text{th}$ Hecke operators, $p \nmid q$, which is denoted by $H_{k,q}^\ast(\chi_0)$, and known as the space of cuspidal newforms of weight $k$, level $q$ and principal central character. We study the $L$-functions associated to these objects and their symmetric lifts. We fix a square-free level $q$ and consider the untwisted family
\begin{align}\label{eq:definitionFrXdeltaq}
\mathcal{F}_{r, X, \delta, q} \ = \ \bigcup_{k < X^{\delta}} \ \on{Sym}^r \left[ H_{k,q}^\ast(\chi_0)\right]
\end{align}
for $\delta > 0$; \textbf{$k$ will always range over even integers in all sums below},  which we denote by a star in the summation. For $p \nmid q$, we define the $p$-local second moment of this family by
\begin{align}\label{eq:definitionofMtwop}
M_{2,p} (\mathcal{F}_{r,X,\delta,q}) \ = \ \frac{1}{\sideset{}{^\ast}\sum_{k < X^{\delta}} \dim H^\ast_{k,q}(\chi_0)} \sideset{}{^\ast}\sum_{k < X^{\delta}} \sum_{f \in H_{k,q}^{\ast}(\chi_0)} \ a_{\on{Sym}^r f}^2(p).
\end{align}
We fix a constant $\sigma > 0$ and sum over primes $p \le X^\sigma$ that do not divide the level, and define the second moment of $\mathcal{F}_{r,X,\delta,q}$ by
\begin{align}\label{eq:definitionofMtwosigma}
M_{2, \sigma}(\mathcal{F}_{r, X, \delta, q}) \ = \ \frac{1}{\pi(X^{\sigma})} \sum_{\substack{p \le X^{\sigma}\\p\neq q}} \ M_{2,p}(\mathcal{F}_{r, X, \delta, q}).
\end{align}
The parameter $\sigma$ controls the number of primes $p$ we sum over compared to the number of weights $k$ we sum over in the $p$-local second moment. In particular, averaging over primes $p$ allows us to exploit the dependence of the coefficients $a^2_{\on{Sym}^rf}(p)$ on the prime $p$; this will extract the lower order terms determining the bias.

We now study the bias of
\begin{align}
M_{2, \sigma} (\mathcal{F}_{r, X, \delta}) \ = \ \lim_{q \to \infty} \ M_{2, \sigma}(\mathcal{F}_{r, X, \delta, q})
\end{align}
where the limit is taken over prime levels $q$. We prove the following bias result for the family $\mathcal{F}_{r,X, \delta}$.

\begin{theorem} \label{theorem:leveltoinfinitysymr}
We have \be M_{2, \sigma} (\mathcal{F}_{r, X, \delta}) \ = \
\left(1 + O\left(X^{-\delta}\right)\right) \left(1 \ + \ \frac{\log \log X^{\sigma}}{\pi(X^\sigma)} \ + \ O\left(\frac1{\pi(X^\sigma)}\right)\right). \ee If we choose $\sigma < \delta$ then the main term of $M_{2, \sigma}(\mathcal{F}_{r, X, \delta})$ is 1 with leading lower order term
\begin{align}
\frac{\log \log X^{\sigma}}{\pi(X^\sigma)}.
\end{align}
We say that these families have a main term of 1 and a small positive bias of $(\log\log X^{\sigma})/\pi(X^\sigma)$ (which tends to 0 as $X\to\infty)$, so long as $\delta > \sigma$.
\end{theorem}

We may also average over the level $q$; as the calculations are similar in the interest of space we will not report on this case (in that situation, $q$ is not square-free so we need to use some results from \cite{BBDDM}). So long as the error terms in counting the forms in the family are smaller than the leading error terms in the Petersson computations we will need, the rate at which we let the weights and levels grow with respect to the rate at which we average over primes does not change the sign of the bias in the family but does change the size of the bias.

Finally, we could also analyze higher moments. The situation is strikingly different here than in the case of elliptic curves, as the Petersson formula is still available, and thus for suitably restricted ranges the computations are handled analogously as above.

\subsection{Preliminaries}

We briefly review some needed facts; see \cite{IK} for a detailed exposition. For an $f \in H_{k,q}^\ast(\chi_0)$, we consider the $L$-function attached to the $r$\textsuperscript{{\rm th}} symmetric lift of $f$, whose local Euler factors are given by
\begin{align}
L_p(\on{Sym}^r f, s) \ = \ \prod_{j = 0}^r \left( 1 - \alpha_p^{r-j} \beta_p^j p^{-s}\right)^{-1}
\end{align} where $\alpha_p, \beta_p$ are the two Satake parameters at $p$. Note if $r=1$ we regain our original form $f$, and thus we can study families of cuspidal newforms and their symmetric lifts simultaneously. We  have \be \lambda_{\on{Sym}^r f}(p) \ = \ \lambda_f(p^r). \label{eq:symtocusp} \ee Before we proceed with computing anything, we recall that a \emph{ramified} prime $p$ is a prime that divides the level of the newform, and an \emph{unramified} prime does not divide the level. Thus, by \eqref{eq:symtocusp} and the theory of $L$-functions attached to holomorphic cusp forms $f$, we have for any distinct, unramified primes \be \lambda_{\on{Sym}^r f}(p) \lambda_{\on{Sym}^r f}(q) \ = \ \lambda_{\on{Sym}^r f}(pq). \ee

To compute the second-moment bias, we use the following standard fact.

\begin{lemma} \label{lemma:secondmomentterm}
For unramified primes $p$,
\begin{align}
\lambda_{\on{Sym}^r f}^2(p) \ = \ \lambda_f^2(p^r) \ = \ \sum_{\ell = 0}^{r} \lambda_f(p^{2 \ell}). \label{eq:secondmomenteigsymr}
\end{align}
\end{lemma}

\begin{proof}
Because $\alpha_p \beta_p = 1$ for $p$ unramified, we have
\begin{align}
\lambda_f^2(p^r) \ = \ \left( \alpha_p^r + \alpha_p^{r-2} + \cdots  + \alpha_p^{-r} \right)^2 \ = \ \sum_{\ell = 1}^{r} (r-\ell + 1) \left( \alpha_p^{2\ell} + \alpha_p^{-2\ell} \right) + (r+1). \label{eq:secondmomenteig}
\end{align}
Since $\lambda_f(p^{m}) \ = \ \alpha_p^m + \alpha_p^{m-2} + \cdots  + \alpha_p^{-m}$, the far RHS of \eqref{eq:secondmomenteig} agrees with the far RHS of \eqref{eq:secondmomenteigsymr}.
\end{proof}

The main tool is the Petersson formula; the version below is Proposition 2.13 in \cite{ILS}.

\begin{prop}[Proposition 2.13, \cite{ILS}] \label{prop:pformula}
Let $q$ be squarefree and $(n, q^2) \mid q$. Then
\begin{align}
\sum_{f \in H_{k, q}^\ast(\chi_0)} \ \lambda_{f}(n) \ = \ \begin{cases}
\frac{k-1}{12} \ \frac{\varphi(q)}{\sqrt{n}} \ + \ O \left( (n, q)^{-\frac{1}{2}} \ n^{\frac16} k^{\frac23} q^{\frac23} \right) & n \emph{ a square and } (n,q) = 1, \\
O \left( (n, q)^{-\frac{1}{2}} \ n^{\frac16} k^{\frac23} q^{\frac23} \right) & \emph{otherwise.}
\end{cases} \label{eq:pformula}
\end{align}
\end{prop}

Lastly, we need asymptotics for even weight $k$ of $\dim H_{k,q}^\ast(\chi_0)$ for $q$ square-free. To do so, we appeal to the following.

\begin{lemma}[Corollary 2.14, \cite{ILS}]\label{lemma:dimcusp}
For even weights $k \geq 2$ and $q$ square-free,
\begin{align}
\dim H_{k,q}^\ast(\chi_0) \ = \ \frac{k-1}{12} \varphi(q) \ + \ O (\left( k q \right)^{2/3}).
\end{align}
\end{lemma}

Note that in particular, the above corollary yields
\begin{align}
\sideset{}{^\ast} \sum_{k < X^{\delta}} \ \dim H_{k,q}^\ast(\chi_0) \ = \ \frac{\varphi(q)}{48} X^{2 \delta} + O\left(X^{5\delta/3} q^{2/3}\right).
\end{align}

\subsection{Proof of Theorem \ref{theorem:leveltoinfinitysymr}}

Let $\chi_0$ denote the principal character. We fix a prime level $q$ and consider the untwisted family
$\mathcal{F}_{r, X, \delta, q}$ defined in \eqref{eq:definitionFrXdeltaq}, and for $\delta, \sigma > 0$ investigate the second moments
$M_{2,p}(\mathcal{F}_{r,X,\delta,q})$ and $M_{2, \sigma}(\mathcal{F}_{r, X, \delta,q})$ defined in equations \eqref{eq:definitionofMtwop} and \eqref{eq:definitionofMtwosigma}. The second moment of interest is the following limit of second moments:
\begin{align}
M_{2, \sigma}(\mathcal{F}_{r, X, \delta}) \ = \ \lim_{q \to \infty} \ M_{2, \sigma}(\mathcal{F}_{r, X, \delta, q}).
\end{align}

Unfolding the RHS of \eqref{eq:definitionofMtwop}, using \eqref{eq:secondmomenteigsymr} and Proposition \ref{prop:pformula}, we have
\begin{align}\label{eq:unfoldedRHSfromMtwo}
& M_{2, \sigma} \left( \mathcal{F}_{r, X, \delta, q} \right) \ = \ \frac1{\pi(X^\sigma)} \frac{1}{\sideset{}{^\ast}\sum_{k < X^{\delta}} \dim(H^{\ast}_{k,q}(\chi_0))} \ \sum_{\substack{p \le X^{\sigma}\\ p\neq q}} \ \sideset{}{^\ast}\sum_{k < X^{\delta}} \ \sum_{f \in H_{k,q}^*(\chi_0) } \lambda_{\on{Sym}^r f}^2(p)\nonumber\\
&\ \ = \ \frac{1}{\pi(X^\sigma)} \frac{1}{\sideset{}{^\ast}\sum_{k < X^{\delta}} \dim(H^{\ast}_{k,q}(\chi_0))} \sum_{\substack{p \leq X^\sigma\\p \neq q}} \sideset{}{^\ast}\sum_{k < X^\delta} \sum_{\ell=0}^r \left(\frac{k-1}{12} \varphi(q) p^{-\ell} + O(p^{\ell/3} k^{2/3} q^{2/3}) \right).
\end{align}

We will now, as well as in later arguments, focus on leading terms and then handle lower order terms at the end. In \eqref{eq:unfoldedRHSfromMtwo} we first fix a prime and investigate the first $k$ sum. Summing $\varphi(q)(k-1)/12$ over the even weight $k$, dividing by the number of such forms, and recalling $\varphi(q) = q-1$ for $q$ prime, we have
\begin{align}
 \frac{X^{2\delta}/48 + O(X^\delta)}{X^{2\delta}/48 + O\left(X^{5\delta/3} q^{2/3}/\varphi(q)\right)} \ = \ 1 + O\left(X^{-\delta} + X^{-\delta/3} q^{-1/3}\right);
\end{align}
the error from counting the number of forms is significantly smaller than the main lower order terms we'll find below.


We now average over the primes $p$. We have
\begin{align}
\sum_{\substack{p \le X^{\sigma}\\p \neq q}} \sum_{\ell = 0}^r p^{-\ell} \ &= \ \sum_{p \le X^{\sigma}} 1 \ + \ \sum_{p \le X^{\sigma}} p^{-1} \ + \ \sum_{\ell = 2}^r \ \sum_{p \le X^{\sigma}} p^{-\ell} + O(1).\label{eq:sumoverprimes}
\end{align}
In particular, from the sum over primes $p$ we extract the two leading terms that diverge as $X \to \infty$ and trivially bound the remaining $O(1)$ terms. The first term on the RHS of \eqref{eq:sumoverprimes} is $\pi(X^\sigma)$. For the second term, we use the (see for example \cite{Dav}):
\begin{align}\label{eq:asymptoticssumreciprocalprimes}
\sum_{p \le X} \ p^{-1} \ = \ \log \log X \ + \ O \left( 1 + \frac{1}{\log X} \right).
\end{align}


Thus, taking into account the normalization factors, the main terms from the Peterrson formula are
\begin{align}
\left(1 + O\left(X^{-\delta} + X^{-\delta/3} q^{-1/3}\right)\right) \left(1 \ + \ \frac{\log \log X^{\sigma}}{\pi(X^\sigma)} \ + \ O\left(\frac1{\pi(X^\sigma)}\right)\right).
\end{align}
Now, accounting for the accumulation of error terms from the Peterrson formula, it suffices to show that, in the limit as $q \to \infty$,
\be
\frac{1}{\pi(X^\sigma)} \frac{1}{\varphi(q) X^{2 \delta}/48 + O(X^{5 \delta/3} q^{2/3})} \sum_{\substack{p \leq X^{\sigma}\\p \neq q}} \sideset{}{^\ast}\sum_{k < X^{\delta}} \sum_{\ell=0}^r O(p^{\ell/3} k^{2/3} q^{2/3})
\ee
is negligible. However, we note that all the big-Oh terms have a factor of $q^{2/3}$, while the second fraction has a factor of $\varphi(q) = q-1$ for $q$ prime. Thus since the big-Oh constants are absolute, in the limit $q \to \infty$, these error terms all vanish; we conclude that
\be
M_{2, \sigma}(\mathcal{F}_{r, X, \delta}) = (1 + O(X^{-\delta}))\left( 1 + \frac{\log \log X^{\sigma}}{\pi(X^{\sigma})} + O\left(\frac{1}{\pi(X^{\sigma})} \right)\right),
\ee
and so it is not hard to see that the leading error term $\frac{\log \log X^{\sigma}}{\pi(X^\sigma)}$ will remain as long as it is not smaller than $O(X^{-\delta})$. For this, it suffices to have $\delta > \sigma$, as claimed. \hfill $\Box$


\section{Convolutions of Families}\label{sec:convolution}

In this section we explore the effect Rankin-Selberg convolution has on biases in second moments; we briefly summarize the framework (see \cite{IK} for additional details). For an automorphic representation $\pi$ on $\text{GL}(n)$, we have the Satake parameters $\{\alpha_{\pi,i}(p)\}_{i=1}^{n}$  as the coefficients in the Euler product of the associated $L$-function
\begin{align}
    L(s, \pi) \ = \ \prod_{p} \prod_{i=1}^{n} (1 -\alpha_{\pi,i}(p)p^{-s})^{-1}.
\end{align}

The Rankin-Selberg method provides a way to combine families of $L$-functions. If the Satake parameters of the $L$-functions for $\pi_{1}, \pi_{2}$ are $\{\alpha_{\pi_{1},i}(p)\}_{i=1}^{n}$ and $\{\alpha_{\pi_{2},j}(p)\}_{j=1}^{m}$, then the pairwise products of the parameters determine the convolved family via
\begin{align}
    \{\alpha_{\pi_{1} \times \pi_{2},k}(p)\}_{k=1}^{nm} \ =\  \{
    \alpha_{\pi_{1},i}(p) \cdot \alpha_{\pi_{2},j}(p)\}_
    {\substack{1 \leq i \leq n \\ 1 \leq j \leq m}}.
\end{align}

Occasionally the resulting $L$-function will not be primitive; for example, if $\pi_1 = \pi_2 = f$ is a cuspidal newform on ${\rm GL}(2)$, then $\zeta(s)$ divides the convolution $L$-function $L(s, f \times f)$, which then factors as $\zeta(s) L(s, {\rm sym}^2 f)$.

We study convolutions of families of Dirichlet $L$-functions as well as $L$-functions associated to cuspidal newforms and elliptic curves. The phenomena we find is that the second moment of the convolution is essentially the product of the moments of the constituent families. However, this does not necessarily enable us to detect a bias in the convolved family, especially if one of the biases is very small. We encounter this scenario in Theorem \ref{theorem:dirichletcuspgrow}.

We describe the families we study. We first consider the convolution of two families of nontrivial Dirichlet characters, say $\mathcal{D}_{q_{1}}, \mathcal{D}_{q_{2}}$, where the levels $q_{1}, q_{2}$ are prime. Since the Satake parameters of these families are the Dirichlet characters, the Dirichlet coefficients of the convolution of these families, denoted by $\mathcal{D}_{q_{1}} \times \mathcal{D}_{q_{2}}$, are $\chi_{1} \chi_{2}$, for $\chi_{1} \in \mathcal{D}_{q_{1}}, \chi_{2} \in \mathcal{D}_{q_{2}}$. We define the second moment of this family as
\begin{align}
	M_{2}(\mathcal{D}_{q_{1}} \times \mathcal{D}_{q_{2}},X) \ = \ \frac{1}{\pi(X)} \sum_{p \le X} \frac{1}{(q_{1} - 2)(q_{2} - 2)} \sum_{\substack{\chi_{1} \in \mathcal{D}_{q_{1}} \\ \chi_{2} \in \mathcal{D}_{q_{2}}}} \chi_{1}^{2}(p) \chi_{2}^{2}(p).
\end{align}

We prove the following bias result for this convolution (see \eqref{eq:defnExqa} for the definition of $E(X,q,a)$).

\begin{theorem} \label{theorem:dirichletconvolved}
	Let $\mathcal{D}_{q_{1}}, \mathcal{D}_{q_{2}}$ be two families of nontrivial Dirichlet characters of distinct prime levels $q_{1}, q_{2}$ with $q_1 \not\equiv \pm 1 \pmod{q_2}, q_2 \not\equiv \pm 1 \pmod{q_1}$. Assuming the Generalized Riemann Hypothesis, the second moment of the convolved family $\mathcal{D}_{q_{1}} \times \mathcal{D}_{q_{2}}$ has main term $\frac{1}{(q_{1}-2)(q_{2} - 2)}$ and lower order term $$\frac{1}{(q_{1} - 2)(q_{2} - 2)}\frac{\sqrt{X}}{\pi(X) \log X} \left(E_{1}(X,q_1,q_2) - E_{2}(X,q_1,q_2)\right)$$ where
\begin{align}
	E_{1}(X,q_1,q_2)\ :=\ E(X, q_{1}q_{2}, 1) + E(X, q_{1}q_{2}, -1)+E(X, q_{1}q_{2}, r_{3}) +E(X, q_{1}q_{2}, r_{4}),
\end{align}
with $r_{3},r_{4}$ being the unique residues satisfying $r_{3} \equiv 1 \bmod{q_{1}}, r_{3} \equiv -1 \bmod{q_{2}}, r_{4} \equiv -1 \bmod{q_{1}}, r_{4} \equiv 1 \bmod{q_{2}}$,
and
\begin{align}
	E_{2}(X,q_1,q_2)\ :=\ E(X, q_{1}, 1) + E(X, q_{1}, -1) + E(X, q_{2}, 1) + E(X, q_{2}, -1).
\end{align}
Additionally assuming GSH, as $X \to \infty$ the bias is sometimes positive and sometimes negative depending on $q_1, q_2$ (and on a logarithmic scale each happens a positive percentage of the time).
\end{theorem}

In the convolution of families of Dirichlet $L$-functions and cuspidal newforms, as before we allow the level for the newforms to grow. We consider the family of nontrivial Dirichlet characters $\mathcal{D}_{q_{1}}$  with prime level $q_{1}$ and the family
\begin{align}
\mathcal{F}_{r, X, \delta, q_{2}} \ = \ \bigcup_{k < X^{\delta}} \ \on{Sym}^r \left[ H_{k,q_{2}}^\ast(\chi_0)\right]
\end{align}
for prime level $q_2$ and $\delta > 0$. In taking the convolution $\mathcal{D}_{q_{1}} \times \mathcal{F}_{r,X,\delta,q_{2}}$ of these families, we find that the Dirichlet coefficients have the form $ \chi(p) \lambda_{\on{Sym}^{r} f}(p)$ for $\chi \in \mathcal{D}_{q_{1}}, f \in H_{k,q_{2}}^{*}(\chi_{0})$. For $\sigma> 0$, we define
\begin{align}
	M_{2,p}(\mathcal{D}_{q_{1}} \times \mathcal{F}_{r, X, \delta, q_{2}}) \ &= \ \frac{1}{q_{1} - 2} \frac{1}{\sideset{}{^\ast}\sum_{k < X^{\delta}} \dim H^{\ast}_{k,q_{2}} (\chi_{0}) } \sideset{}{^\ast}\sum_{k < X^{\delta}}\sum_{\substack{\chi \in \mathcal{D}_{q_{1}}\\ f \in H_{k,q_{2}}^{*}(\chi_{0}) }}  \chi^{2}(p) \lambda^{2}_{\on{Sym}^{r}f} (p) \nonumber\\
	M_{2, \sigma}(\mathcal{D}_{q_{1}} \times \mathcal{F}_{r,X,\delta}) \ &= \ \lim_{q_{2} \rightarrow \infty} \frac{1}{\pi(X^{\sigma})} \sum_{\substack{p \le X^{\sigma}\\p \neq q_2}} M_{2,p}(\mathcal{D}_{q_{1}} \times \mathcal{F}_{r, X, \delta, q_{2}}),
\end{align}
where as before $*$ indicates that the sum is taken over even $k$. For this family, we prove the following bias result.

\begin{theorem} \label{theorem:dirichletcuspgrow}
	Let $\mathcal{D}_{q_{1}}$ be the family of nontrivial Dirichlet characters of prime level $q_{1}$ and let $\mathcal{F}_{r,X,\delta,q_{2}}$ be the family of the $r$\textsuperscript{{\rm th}} symmetric lifts of cuspidal newforms with even weight $k<X^{\delta}$ and prime level $q_{2}$. Assuming GRH and $\delta > \sigma$, the second moment of the convolved family $\mathcal{D}_{q_{1}} \times \mathcal{F}_{r,X,\delta,q_{2}}$ as $q_{2} \to \infty$ has main term $\frac{1}{q_{1}-2}$ and lower order terms
	\begin{align}
		\frac{1}{(q_{1} - 2) \pi(X^{\sigma})}  \left( \frac{\sqrt{X^{\sigma}}}{\log X^{\sigma}} (E(X^{\sigma}, q_{1}, 1) + E(X^{\sigma},q_{1},-1)) + \log \log X^{\sigma} \right) + O_{q_1, r}(1/\pi(X^\sigma)),
	\end{align}
	where the subscripts denote that the implied constant depends on $q_1, r$.
\end{theorem}
Unfortunately, with the second moment expressed in the above form, we do not know how to say anything precise about the sign or size of the leading lower order term.\\

We may also convolve families of cuspidal newforms with each other, allowing the levels to grow. We consider families of cuspidal newforms
\begin{align}
\mathcal{F}_{r_{1}, X, \delta_{1}, q_{1}} \ =\ \bigcup_{k_{1}<X^{\delta_{1}}} \on{Sym}^{r} \left[H_{k_{1},q_{1}}^{*}(\chi_{0}) \right],\  \mathcal{F}_{r_{2}, X, \delta_{2}, q_{2}} \ =\ \bigcup_{k_{2}<X^{\delta_{2}}} \on{Sym}^{r} \left[H_{k_{2},q_{2}}^{*}(\chi_{0}) \right],
\end{align}
where the family $\mathcal{F}_{r_{1}, X, \delta_{1}, q_{1}} \times \mathcal{F}_{r_{2}, X, \delta_{2}, q_{2}}$ has Dirichlet coefficients given by $\lambda_{\on{Sym}^{r_{1}}f_{1}}(p) \lambda_{\on{Sym}^{r_{2}}f_{2}}(p)$ for $f_{1} \in H_{k_{1}, q_{1}}^{*}(\chi_{0}), f_{2} \in H_{k_{2}, q_{2}}^{*}(\chi_{0}).$ We take the $p$-local second moment of the convolved family, $M_{2, p} ( \mathcal{F}_{r_{1}, X, \delta_{1}, q_{1}} \times \mathcal{F}_{r_{2}, X, \delta_{2},q_{2}} )$ to be
\begin{align}
	\left(\sideset{}{^\ast}\sum_{k_{1} < X^{\delta_{1}}} \dim H_{k_{1}, q_{1}}^{*}(\chi_{0}) \sideset{}{^\ast}\sum_{k_{2} < X^{\delta_{2}}} \dim H_{k_{2}, q_{2}}^{*}(\chi_{0})\right)^{-1}\sum_{k_{1} < X^{\delta_{1}}} \ \sum_{k_{2} < X^{\delta_{2}}} \ \sum_{\substack{f_{1} \in H_{k_{1}, q_{1}}^{*}(\chi_{0})\\ f_{2} \in H_{k_{2}, q_{2}}^{*}(\chi_{0})}} \lambda_{\on{Sym}^{r_{1}}f_{1}}^{2}(p) \lambda_{\on{Sym}^{r_{2}}f_{2}}^{2}(p),
\end{align}
and
\begin{align}
	M_{2, \sigma}(\mathcal{F}_{r_{1}, X, \delta_{1}} \times \mathcal{F}_{r_{2}, X, \delta_{2}}) = \lim_{q_{1}, q_{2} \to \infty} \frac{1}{\pi(X^{\sigma})} \sum_{\substack{p \le X^{\sigma}\\ p \neq q_1, q_2}} M_{2,p}(\mathcal{F}_{r_{1}, X, \delta_{1}, q_{1}} \times \mathcal{F}_{r_{2}, X, \delta_{2}, q_{2}}).
\end{align}
For this convolution, we derive the following bias result.

\begin{theorem} \label{theorem:convcuspgrow}
	Let $\mathcal{F}_{r_{1},X, \delta_{1}, q_{1}}, \mathcal{F}_{r_{2}, X, \delta_{2}, q_{2}}$ be families of $r_{1}$\textsuperscript{{\rm th}} and $r_{2}$\textsuperscript{{\rm th}} symmetric lifts of cuspidal newforms with even weights $k_{1} < X^{\delta_{1}}, k_{2} < X^{\delta_{2}}$ and square-free distinct levels $q_{1}, q_{2}$. Assuming that $\sigma < \min(\delta_{1}, \delta_{2})$, we have
\begin{align}
    M_{2,\sigma} ( \mathcal{F}_{r_{1}, X, \delta_{1}} \times \mathcal{F}_{r_{2}, X, \delta_{2}} )\ =\ 1 + \frac{2 \log \log X^{\sigma}}{\pi(X^{\sigma})} + O(X^{-\sigma} \log X^{\sigma}),
\end{align}
which has positive bias.
\end{theorem}

\subsection{Proof of Theorem \ref{theorem:dirichletconvolved}}
We ignore the normalization factors for now, as they are easily incorporated later. The quantity of primary interest is
\begin{align}
	\sum_{p \le X} \sum_{\substack{\chi_{1} \in \mathcal{D}_{q_{1}} \\ \chi_{2} \in \mathcal{D}_{q_{2}}}} \chi_{1}^{2}(p) \chi_{2}^{2}(p) \ &= \ \sum_{p \le X}\left( \sum_{\chi_{1} \in \mathcal{D}_{q_{1}}} \chi_{1}^{2}(p) \right) \left(\sum_{\chi_{2} \in \mathcal{D}_{q_{2}}} \chi_{2}^{2}(p) \right),
\end{align}
which orthogonality relations from \eqref{eq:schurorthog} allow us to rewrite as
\begin{align}
	\sum_{\substack{p \le X\\ p \equiv \pm 1 (q_{1})\\ p \equiv \pm 1 (q_{2})}} (q_{1} - 1)(q_{2} - 1) \ - \ \sum_{\substack{p \le X\\ p \equiv \pm 1 (q_{1})}} (q_{1} - 1) \ - \ \sum_{\substack{p \le X\\ p \equiv \pm 1 (q_{2})}} (q_{2} - 1) \ + \ \sum_{p \le X} 1.
\end{align}
Now, the definition of $E(X, q, a)$, given in \eqref{eq:defnExqa}, allows us to simplify this further as
\begin{align}
	\pi(X) + \frac{\sqrt{X}}{\log X} \left( E_{1}(X,q_1,q_2) - E_{2}(X,q_1,q_2) \right)
\end{align}
with $E_{1}, E_{2}$ as defined in the statement of Theorem \ref{theorem:dirichletconvolved}. Now, given $q_2 \not\equiv 1 \pmod{q_1}$,
\begin{align}
    E(X, q_1, 1) &\ = \ \frac{1}{q_2 - 1} ( \varphi(q_1 q_2) \pi(X, q_1, 1) - (q_2 - 1)\pi(X)) \frac{\log X}{\sqrt{X}} \nonumber \\
    &\ =\ \frac{1}{q_2 - 1} \left( \varphi(q_1 q_2) \left(\sum_{\substack{0 \leq k < q_2\\ k q_1 \not \equiv -1 (q_2)}} \pi(X, q_1 q_2, 1 + k q_1)\right) - (q_2 - 1) \pi(X) \right) \frac{\log X}{\sqrt{X}}\nonumber \\
    &\ = \ \frac{1}{q_2 - 1}\sum_{\substack{0 \leq k < q_2\\ k q_1 \not\equiv -1 (q_2)}} E(X, q_1 q_2, 1 + k q_1),
\end{align}
where we ignore the value of $k$ that satisfies $k q_1 \equiv -1  \pmod{q_2}$ since in that case $\pi(X, q_1 q_2, 1 + k q_1)$ counts primes that are necessarily divisible by $q_2$, since if $q_2 \mid 1 + k q_1$ then $p \equiv 1 + k q_1 \pmod{q_1 q_2}$ means that $q_2 \mid p$, and so $p = q_2$. However, we assumed that $q_2$ is not congruent to $1 \pmod{q_1}$, and so there are no primes counted by $\pi(X, q_1 q_2, 1 + k q_1)$ in this case.\\

We can compare $E(X,q_1, 1)$ and $E(X, q_1 q_2, 1)$. The former is an average of values, one of which is precisely $E(X, q_1 q_2, 1)$; analogous calculations can be done to compare $E(X, q_1,-1)$ and $E(X, q_1 q_2, -1)$ and the other two pairs of terms. Writing everything over the denominator $q_2-1$, it becomes clear that if the four residue classes (modulo $q_{1} q_{2}$) with the most primes correspond to the terms in $E_{1}$, then the bias is positive; on the other hand, those residue classes can be those with the fewest primes, yielding a negative bias. \hfill \qed


\subsection{Proof of Theorem \ref{theorem:dirichletcuspgrow}}
We separate the calculation into several steps, starting with $M_{2,p}(\mathcal{D}_{q_{1}} \times \mathcal{F}_{r, X, \delta, q_{2}})$. Note that in
\bea & &	M_{2,p}(\mathcal{D}_{q_{1}} \times \mathcal{F}_{r, X, \delta, q_{2}})\nonumber\\  & & =\ \frac{1}{q_{1} - 2} \frac{1}{\sideset{}{^\ast}\sum_{k < X^{\delta}} \dim H_{k, q_{2}}^{*}(\chi_{0})}\sideset{}{^\ast}\sum_{k < X^{\delta}} \left( \sum_{\chi \in \mathcal{D}_{q_{1}}} \chi^{2}(p) \right) \left( \sum_{f \in H_{k,q_{2}}^{*}(\chi_{0})} \lambda_{\on{Sym}^{r} f}^{2}(p) \right) \ \ \eea
the sum $\sum_{\chi \in \mathcal{D}_{q_{1}}} \chi^{2}(p)$ has no dependence on $k$. Applying the Petersson formula in Proposition \ref{prop:pformula} (since we assume $p \neq q_2$) and expanding, we find
\begin{align}\label{eq:vanishingerror}
	\sideset{}{^\ast}\sum_{k < X^{\delta}} \sum_{f \in H_{k,q_{2}}^{*}(\chi_{0})} \lambda_{\on{Sym}^{r} f}^{2} (p)\ &=\ \sideset{}{^\ast}\sum_{k < X^{\delta}} \sum_{f \in H_{k,q_{2}}^{*}(\chi_{0})} \ \sum_{\ell=0}^{r} \lambda_{f}(p^{2\ell}) \nonumber\\ \ &= \ \varphi(q_{2})\ \sideset{}{^\ast} \sum_{k < X^{\delta}} \sum_{\ell=0}^{r} \left[\frac{k-1}{12} p^{-\ell} + O(p^{\frac{\ell}{3}} k^{\frac{2}{3}}q_{2}^{\frac{2}{3}} / \varphi(q_{2}))\right].
\end{align}
Gathering the error terms, we can write this as
\be
	\varphi(q_{2}) \left(\sideset{}{^\ast}\sum_{k< X^{\delta}} \frac{k-1}{12} \right)\left(\sum_{\ell=0}^{r} p^{-\ell} \right) + O\left(r p^{\frac{r}{3}} X^{\frac{5\delta}{3}} q_{2}^{\frac{2}{3}}\right).
\ee
To normalize by $\sideset{}{^\ast} \sum_{k < X^{\delta}}^{} \dim H_{k,q_{2}}^{*}(\chi_{0})$, we recall our earlier calculation which showed the ratio of the sum $\varphi(q_{2})(k-1)/12$ over even weight $k < X^{\delta}$ to the sum $\sideset{}{^\ast}\sum_{k < X^{\delta}} \dim H_{k, q_{2}}^{*}(\chi_{0})$ is
\bea
\frac{X^{2\delta}/48 + O(X^{\delta})}{X^{2\delta}/48 + O(X^{5 \delta/3} q_{2}^{2/3}/ \varphi(q_{2}))} \   =\ 1 + O\left(X^{-\delta} + X^{-\delta/3} q_{2}^{-1/3} \right).
\eea
Thus, we have that $M_{2,p}(\mathcal{D}_{q_{1}} \times \mathcal{F}_{r, X, \delta, q_{2}})$ is
\bea
\frac{1}{q_{1} - 2} \left(\sum_{\chi \in \mathcal{D}_{q_{1}}}^{} \chi^{2}(p) \right)\left[\left( \sum_{\ell = 0}^{r} p^{-\ell} \right)\left(1 + O\left(X^{-\delta} + X^{-\frac{\delta}{3}} q_{2}^{-\frac13} \right)\right)  + O\left(\frac{r p^{\frac{r}{3}} X^{\frac{-\delta}{3}} q_{2}^{\frac{2}{3}}}{\varphi(q_{2})}\right)\right], \nonumber\\
\eea
where we also divide the error term by $\varphi(q_2)X^{2 \delta}$. 
Now, summing over primes, we consider the sum of one of the error terms
\bea
\sum_{\substack{p \leq X^{\sigma}\\p \neq q_2}}^{} \left( \sum_{\chi \in \mathcal{D}_{q_1}} \chi^2(p) \right) O(r p^{\frac{r}{3}} X^{\frac{-\delta}{3}} q_{2}^{\frac{2}{3}}/\varphi(q_{2})) \ &= \  O\left(\phi(q_1) r X^{\frac{(r+3)\sigma - \delta}{3}} q_{2}^{-\frac{1}{3}}\right),
\eea
where we again use the fact that $\varphi(q_2) = q_2-1$ for $q_2$ prime. The main terms are given by
\begin{align}
	&\frac{1}{q_{1} - 2}(1+ O(X^{-\delta} + X^{-\delta/3} q_{2}^{-1/3})) \sum_{\substack{p \leq X^{\sigma}\\ p \neq q_2}}^{} \left( \sum_{\chi \in \mathcal{D}_{q_{1}}}^{} \chi^{2}(p) \right) \left(\sum_{\ell=0}^{r} p^{-\ell}\right),
\end{align}
and splitting the sum over $\ell$ we write the nested sums as
\begin{align}
	& \sum_{\substack{p \le X^{\sigma}\\ p \equiv \pm 1 (q_{1})\\p \neq q_2}} (q_{1}-1) \ -\ \sum_{\substack{p \le X^{\sigma}\\ p \neq q_2}} 1 \ +\ \sum_{\substack{p \le X^{\sigma}\\ p \neq q_2}}\ \left(\sum_{\chi \in \mathcal{D}_{q_{1}}} \chi^{2}(p) \right) \left(\sum_{\ell=1}^{r} p^{-\ell} \right).
\end{align}
The restriction $p \neq q_2$ changes the above sums by at most a constant, and $q_1$ is fixed, so we can rewrite the above in terms of $E(X^{\sigma}, q_{1}, 1)$ and $E(X^{\sigma}, q_{1}, -1)$ as in the case of Dirichlet $L$-functions:
\begin{align}
	\pi(X^{\sigma}) \ + \ &\frac{\sqrt{X^{\sigma}}}{\log X^{\sigma}} (E(X^{\sigma}, q_{1}, 1) + E(X^{\sigma}, q_{1}, -1)) + (q_{1} - 1)\sum_{\substack{p \leq X^{\sigma} \\ p \equiv \pm 1 (q_{1})}} \frac{1}{p} - \sum_{p \leq X^{\sigma}} \frac{1}{p} +  O(r +q_1 + q_1r)\nonumber\\
	\ &= \ \pi(X^{\sigma}) + \frac{\sqrt{X^{\sigma}}}{\log X^{\sigma}} (E(X^{\sigma}, q_{1}, 1) + E(X^{\sigma}, q_{1}, -1)) + \log \log X^{\sigma} + O_{q_1, r}(1),
\end{align}
where we use the well-known corollary of the Prime Number Theorem for arithmetic progressions,
\begin{align}
	\sum_{\substack{p \leq X^{\sigma} \\ p \equiv a \bmod{q_{1}}}} \frac{1}{p}\ =\ \frac{\log \log X^{\sigma}}{\varphi(q_{1})} + O(\phi(q_1)).
\end{align}
Thus, the second moment $M_{2, \sigma}(\mathcal{D}_{q_{1}} \times \mathcal{F}_{r, X, \delta})$ is
\begin{eqnarray} & & \lim_{q_{2} \to \infty} \frac{1}{q_{1} - 2}\frac{(1 + O(X^{-\delta} + X^{-\delta/3} q_{2}^{-1/3}))}{\pi(X^{\sigma})} \left(  \pi(X^{\sigma}) + \frac{\sqrt{X^{\sigma}}}{\log X^{\sigma}} (E(X^{\sigma}, q_{1}, 1)+ E(X^{\sigma}, q_{1}, -1)) \right.\nonumber\\ & & \ \ \ +\  \left. \log \log X^{\sigma} + O_{q_1,r}(1) + O\left(\phi(q_1)r X^{\frac{(r+3)\sigma - \delta}{3}} q_{2}^{-\frac{1}{3}}\right)\bigg)\right. . \end{eqnarray}

The main term is $1/(q_1-2)$, and as $q_{2} \to \infty$ we find that $O(X^{-\delta} + X^{-\delta/3} q_{2}^{-1/3})$ goes to $O(X^{-\delta})$ and $O(\phi(q_1) r X^{\frac{(r+3)\sigma - \delta}{3}} q_{2}^{-1/3})$ vanishes since the implied constants are absolute. We are left with
\begin{align}
	&\frac{1 + O(X^{-\delta})}{(q_{1}-2) \pi(X^{\sigma})} \left( \pi(X^{\sigma}) + \frac{\sqrt{X^{\sigma}}}{\log X^{\sigma}} (E(X^{\sigma}, q_{1}, 1) + E(X^{\sigma}, q_{1}, -1)) + \log \log X^{\sigma} + O_{q_1, r}(1)\right)\nonumber\\
	&= \ \frac{1}{q_{1} - 2} + \frac{1}{(q_{1} - 2) \pi(X^{\sigma})} \left( \frac{\sqrt{X^{\sigma}}}{\log X^{\sigma}} (E(X^{\sigma}, q_{1}, 1) + E(X^{\sigma}, q_{1}, -1)) \ +\ \log \log X^{\sigma} \right)\nonumber \\ & \ +\ O_{q_1, r}(X^{-\delta}+ \pi(X^\sigma)^{-1}).
\end{align}
To conclude, we note that since $\delta > \sigma$, the error term $O_{q_1, r}(X^{-\delta} + \pi(X^\sigma)^{-1})$ can be consolidated as $O_{q_1, r}(1/\pi(X^\sigma))$. \hfill \qed

\subsection{Proof of Theorem \ref{theorem:convcuspgrow}}
In the $p$-local second moment $M_{2,p}(\mathcal{F}_{r_{1}, X, \delta_{1}, q_{1}} \times \mathcal{F}_{r_{2}, X, \delta_{2}, q_{2}})$, the entire expression factors as
\begin{align}
	\prod_{i=1}^{2} \left( \frac{1}{\sideset{}{^\ast}\sum_{k_{i} < X^{\delta_{i}}}  \dim H_{k_{i}, q_{i}}^{*}(\chi_{0})} \ \sideset{}{^\ast}\sum_{k_{i} < X^{\delta_{i}}} \sum_{f_i \in H_{k_i, q_i}^*(\chi_0)}\lambda_{\on{Sym}^{r_{i}}f_{i}}^{2}(p)\right),
\end{align}
and applying the Petersson formula yields, that the above is equal to
\begin{align}
	\prod_{i=1}^{2} \left( \frac{\varphi(q_{i})}{\sideset{}{^\ast}\sum_{k_{i} < X^{\delta_{i}}} \dim H_{k_{i},q_{i}}^{*}(\chi_{0})} \sideset{}{^\ast}\sum_{k_{i} < X^{\delta_{i}}} \frac{k_{i} - 1}{12} \sum_{\ell_{i}=0}^{r_{i}} [p^{-\ell_{i}} + O(p^{\frac{\ell_{i}}{3}} k_{i}^{-\frac{1}{3}} q_{i}^{\frac{2}{3}}/\varphi(q_{i}))] \right).
\end{align}
Continuing the calculations analogously to the proof of Theorem \ref{theorem:leveltoinfinitysymr}, we find that it suffices to calculate the leading terms of
\begin{align}
	\frac{1 + O(X^{-\delta_{1}} + X^{-\delta_{2}})}{\pi(X^{\sigma})}\sum_{\substack{p \le X^{\sigma}\\p \neq q_1, q_2}} \left(\sum_{\ell_{1}=0}^{r_{1}} p^{-\ell_{1}} \right) \left(\sum_{\ell_{2}=0}^{r_{2}} p^{-\ell_{2}}\right).
\end{align}
We claim those are
\begin{align}
	\frac{1 + O(X^{-\delta_{1}} + X^{-\delta_{2}})}{\pi(X^{\sigma})}\left(\sum_{\substack{p \le X^{\sigma}\\p \neq q_1, q_2}} 1 + \frac{2}{p}\right)\ =\ (1 + O(X^{-\delta_{1}} + X^{-\delta_{2}}))\left(1 + \frac{2 \log \log X^{\sigma}}{\pi(X^{\sigma})} + O\left(\frac{1}{\pi(X^{\sigma})}\right)\right).
\end{align}
It is trivial to bound the other terms from
\be
\sum_{\substack{p \leq X^{\sigma}\\ p \neq q_1, q_2}} \left(\sum_{\ell_{1}=0}^{r_{1}} p^{-\ell_{1}} \right) \left(\sum_{\ell_{2}=0}^{r_{2}} p^{-\ell_{2}}\right)
\ee
other than the term $\sum_{p \leq X^{\sigma}} (1 + \frac{2}{p})$ by $O(1)$, and so we conclude that as long as $\sigma < \min(\delta_{1}, \delta_{2})$, the leading error term will be $\frac{2 \log \log X^{\sigma}}{\pi(X^{\sigma})}$, yielding the claimed bias. \hfill \qed


\section{Conclusion}

The motivation of this paper is to understand how the finer arithmetic of families of $L$-functions affect the distribution of the zeros near the central point. The Katz-Sarnak Density Conjecture \cite{KaSa1,KaSa2} has been verified for many families for suitably restricted test functions, but these results only concern the main term. Similar to the work of Rudnick and Sarnak \cite{RudSa}, who showed that the universality in the $n$-level correlations comes from the universality in the main term of moments of the Satake parameters, the main term of the moments of the Satake parameters also leads to the small number of observed main terms for the $n$-level densities (specifically, unitary, symplectic and orthogonal behavior). Differences in families are not seen in the main terms in these statistics; however, we can observe differences in behavior arising from lower order terms in these moments.

In this paper we look at several families. The most interesting are the one-parameter elliptic curves, where a negative bias is always seen in the second moments. This has implications for the distribution of zeros; in particular, it leads to a small correction term which helps explain some of the observed ``excess rank'' in families with small conductors. We hope that this will encourage further work studying whether similar negative biases exist in general for other families. If the negative biases persist, it is natural to wonder if there is a deeper explanation for this phenomenon.


\appendix

\section{Second Moments of Linear Elliptic Curve Families}\label{sec:linearellipticfamilies}



\subsection{Family \texorpdfstring{$\mathcal{E}(T): y^2 =  (ax^2+bx+c)(dx+e+T)$}{E(T): y2 =  (ax2+bx+c)(dx+e+T)}}


\begin{proposition}
The one-parameter family
\begin{align}
\mathcal E(T): y^2 & \ = \  (ax^2+bx+c)(dx+e+T)
\end{align}
with $p>3$ prime, $a,b,c,d,e\in\Z$ and $p\nmid a,d$ has vanishing first moment, hence rank zero, and second moment given by
\begin{align}
A_{2,\mathcal E}(p) & \ = \ \begin{cases}
p^2-\left(1+\legendre{b^2-4ac}{p}\right)p - 1 &{\rm if\ } p\nmid b^2-4ac\\
p - 1 &{\rm if\ } p\mid b^2-4ac.
\end{cases}
\end{align}
\end{proposition}

\begin{proof}
We have $P(x)=ax^2+bx+c$ and $Q(x)=P(x)(dx+e)$. Substituting into \eqref{eq:expansionlinearellcurvesum},
\begin{align}
A_{2,\mathcal E}(p) & \ = \ p\left[\sum_{P(x)\equiv 0} \legendre{Q(x)}{p}\right]^2 - \left[\sum_{x\;(p)} \legendre{P(x)}{p}\right]^2 + p\sum_{\Delta(x,y)\equiv 0} \legendre{P(x)P(y)}{p} \nonumber\\
& \ = \ p\cdot0 - \left[\legendre{a}{p}\cdot\begin{cases}
-1 &{\rm if\ } p\nmid b^2-4ac\\
p-1&{\rm if\ } p\mid b^2-4ac
\end{cases}\right]^2 + p\sum_{\Delta(x,y)\equiv 0} \legendre{P(x)P(y)}{p} \nonumber\\
& \ = \  \begin{cases}
-1 &{\rm if\ } p\nmid b^2-4ac\\
-(p-1)^2&{\rm if\ } p\mid b^2-4ac
\end{cases}\bigg\} \ + \ p\sum_{\Delta(x,y)\equiv 0} \legendre{P(x)P(y)}{p}
\end{align}
since $p\nmid a$. Note that $\Delta(x,y) = \big(P(y)Q(x) - P(x)Q(y)\big)^2 \equiv 0$ if and only if $P(x)\equiv0$, $P(y)\equiv0$, or $x\equiv y$, since
\begin{align}
P(y)Q(x) - P(x)Q(y) & \ = \ (ax^2+bx+c)(ay^2+by+c)[(dx+e)-(dy+e)] \nonumber\\
& \ = \ P(x)P(y)\,d(x-y).
\end{align}
Thus the sum over $\Delta(x,y) \equiv 0$ becomes
\begin{align}
\sum_{\Delta(x,y)\equiv 0} \legendre{P(x)P(y)}{p} & = \sum_{x\equiv y}\legendre{P(x)P(y)}{p} = \sum_{x\;(p)}\legendre{P(x)}{p}^2 \nonumber\\
& = p-\#\{\alpha:P(\alpha)\equiv0\;(p)\}.
\end{align}
Then since $\#\{\alpha:P(\alpha)\equiv0\;(p)\} = 1+\legendre{b^2-4ac}{p}$, we have
\begin{align}
A_{2,\mathcal E}(p) \ & = \ \begin{cases}
-1 &{\rm if\ } p\nmid b^2-4ac\\
-(p-1)^2&{\rm if\ } p\mid b^2-4ac
\end{cases}\bigg\}  + p\left(p - \left(1+\legendre{b^2-4ac}{p}\right)\right)\nonumber\\
& = \ \begin{cases}
p^2-\left(1+\legendre{b^2-4ac}{p}\right)p - 1 &{\rm if\ } p\nmid b^2-4ac\\
p - 1 &{\rm if\ } p\mid b^2-4ac
\end{cases}
\end{align}
as claimed.
\end{proof}


\subsection{Family \texorpdfstring{$\mathcal E(T): y^2 =  x(ax^2+bx+c+dTx)$}{E(T): y2 =  x(ax2+bx+c+dTx)}}


\begin{proposition}
The family
\begin{align}
\mathcal E(T): y^2 & \ = \  x(ax^2+bx+c+dTx)
\end{align}
with $p>3$ prime, $a,b,c,d\in\Z$ and $p\nmid a,d$ has vanishing first moment, hence rank zero, and second moment given by
\begin{align}
A_{2,\mathcal E}(p) \  = \ \begin{cases}
p^2 - 2 p - 1 &{\rm if\ } p\nmid c\\
p - 1 &{\rm if\ } p\mid c.
\end{cases}
\end{align}
\end{proposition}

\begin{proof}
We have $P(x)=dx^2$ and $Q(x)=x(ax^2+bx+c)$. Substituting into \eqref{eq:expansionlinearellcurvesum},

\begin{align}
A_{2,\mathcal E}(p) & \ = \ p\left[\sum_{P(x)\equiv 0} \legendre{Q(x)}{p}\right]^2 - \left[\sum_{x\;(p)} \legendre{P(x)}{p}\right]^2 + p\sum_{\Delta(x,y)\equiv 0} \legendre{P(x)P(y)}{p} \nonumber\\
& \ = \ p\cdot0 - (p-1)^2 + p\sum_{\Delta(x,y)\equiv 0} \legendre{P(x)P(y)}{p}.
\end{align}

Note that $\Delta(x,y) \equiv 0$ if and only if $x\equiv0$, $y\equiv0$, $x\equiv y$, or $axy\equiv c$ since
\begin{align}
P(y)Q(x) - P(x)Q(y) & \ = \ dxy[y(ax^2+bx+c)-x(ay^2+by+c)] \nonumber\\
& \ = \ dxy(x-y)(axy-c).
\end{align}
But $P(x)P(y) \equiv 0$ when $xy\equiv0$, so by inclusion-exclusion we have
\begin{align}
\sum_{\Delta(x,y)\equiv 0} \legendre{P(x)P(y)}{p} &
\ = \ \left(\sum_{x\equiv y \not\equiv0}+\sum_{axy\equiv c\atop x,y\not\equiv0}-\sum_{ax^2\equiv c\atop x\not\equiv0 \text{ and } x \equiv y}\right)\legendre{P(x)P(y)}{p} \nonumber\\
& \ = \ \sum_{x\not\equiv0}\left(\legendre{P(x)}{p}^2 + \legendre{P(x)P(c/ax)}{p}\right) - 2 \legendre{d(c/a)^2}{p}^2 \nonumber\\
& \ = \sum_{x\not\equiv0}\left(\legendre{dx^2}{p}^2 + \legendre{(cd/a)^2}{p}\right) - 2 \legendre{c^2}{p}  = \sum_{x\not\equiv0}\left(1+\legendre{c^2}{p}\right) - 2 \legendre{c^2}{p}\nonumber\\
& \ =\  (p - 1) \left(1+\legendre{c^2}{p}\right)\ - \ 2 \legendre{c^2}{p}
\end{align}
by the assumption $p\nmid a,d$. Hence we obtain
\begin{align}
A_{2,\mathcal E}(p) & \ = \ -p^2+2p-1 + p\sum_{\Delta(x,y)\equiv 0} \legendre{P(x)P(y)}{p}  \ = \ (p^2-3p)\legendre{c^2}{p} + p - 1.
\end{align}
This gives the result.
\end{proof}


\subsection{Family \texorpdfstring{$\mathcal E(T): y^2 = x(ax+b)(cx+d+Tx)$}{ E(T): y2 = x(ax+b)(cx+d+Tx)}}


\begin{proposition}
The one-parameter family
\begin{align}
\mathcal E(T): y^2 & \ = \  x(ax+b)(cx+d+Tx)
\end{align}
with $p>3$ prime, $a,b,c,d\in\Z$ and $p\nmid a,d$ has vanishing first moment, hence rank zero, and second moment given by
\begin{align}
A_{2,\mathcal E}(p) \  = \ \begin{cases}
p^2-2p - 1 &{\rm if\ } p\nmid b\\
p^2 - p &{\rm if\ } p\mid b.
\end{cases}
\end{align}
\end{proposition}

\begin{proof}
We have $P(x)=x^2(ax+b)$ and $Q(x)=x(ax+b)(cx+d)$. Noting that $P(x)\equiv 0$ implies $Q(x)\equiv 0$, and $\sum_{x\;(p)} \legendre{P(x)}{p} = \sum_{x\not\equiv0} \legendre{ax+b}{p}=-\legendre{b}{p}$ by Lemma \ref{lem:sumlegendrelinquad}. Then substituting into \eqref{eq:expansionlinearellcurvesum}, we have
\begin{align}
A_{2,\mathcal E}(p) & \ = \ p\left[\sum_{P(x)\equiv 0} \legendre{Q(x)}{p}\right]^2 - \left[\sum_{x\;(p)} \legendre{P(x)}{p}\right]^2 + p\sum_{\Delta(x,y)\equiv 0} \legendre{P(x)P(y)}{p} \nonumber\\
& \ =  p\cdot0 \ -\legendre{b}{p}^2 \ + \
p\sum_{\Delta(x,y)\equiv 0} \legendre{P(x)P(y)}{p}.
\end{align}

Note that $\Delta(x,y) \equiv 0$ if and only if $x\equiv0$, $y\equiv0$, $ax\equiv -b$, $ay\equiv -b$, or $x\equiv y$ since
\begin{align}
P(y)Q(x) - P(x)Q(y)  \ & = \ xy(ax+b)(ay+b)\left(y(cx+d)-x(cy+d)\right) \nonumber\\
& \ = \ xy(ax+b)(ay+b)d(y-x).
\end{align}
All cases except $x\equiv y$ imply $P(x)P(y)\equiv0$, so
\begin{align}
\sum_{\Delta(x,y)\equiv 0} \legendre{P(x)P(y)}{p} & \ = \ \sum_{x\equiv y} \legendre{P(x)P(y)}{p} \ = \ \sum_{x\;(p)} \legendre{P(x)}{p}^2 \ = \ \sum_{x\not\equiv0} \legendre{ax+b}{p}^2 \nonumber\\
& = \ \sum_{x\;(p)} \legendre{ax+b}{p}^2 - \legendre{b}{p}^2
 \ = \ p-1 - \legendre{b^2}{p},
\end{align}
recalling $p\nmid a$. Hence plugging back in, we obtain
\begin{align}
A_{2,\mathcal E}(p) & \ = p\sum_{\Delta(x,y)\equiv 0} \legendre{P(x)P(y)}{p}\ - \ \legendre{b^2}{p}  \ = \ (p^2-p) \ - \ (p+1)\legendre{b^2}{p}.
\end{align}
This gives the result.
\end{proof}

\section{Additional Elliptic Curve Families}\label{sec:michellewu}

\noindent By Steven J. Miller and Jiefei Wu \\ \

Miller and Wu \cite{Wu} investigated additional one-parameter, and some two-parameter, families of elliptic curves. We briefly summarize a representative sample of their results. While in Table \ref{table:michelleone} most of the families were carefully chosen so that the corresponding Legendre sums could be evaluated exactly, the last two were deliberately chosen differently in order to provide a better test of the Bias Conjecture. Note the normalization in \cite{Wu} includes dividing the first and second moments by $p$ to be a true average, and we set $\delta_1(p)$ to be $1$ if $p \equiv 1 \bmod 4$ and 0 otherwise.

\begin{table}[htb]
\resizebox{\textwidth}{!}{%
\begin{tabular}{|l|l|l|l|}
\hline
One-Parameter Family     & Rank & $pA_{1,\mathcal{F}(p)}$ & $pA_{2,\mathcal{F}(p)}$                             \\ \hline
$y^2 = x^3-x^2-x+t$      & 0    & 0                      & $p^2-2p-\js{-3}p$                                  \\ \hline
$y^2=x^3-tx^2+(x-1)t^2$  & 0    & 0                     & $p^2-2p-[\sum_{x(p)}\js{x^3-x^2+x}]^2-\js{-3}p$ \\ \hline
$y^2=x^3+tx^2+t^2$       & 1    & $-p$                      & $p^2-2p-\js{-3}p-1$        \\ \hline
$y^2=x^3+tx^2+x+1$       & 1    & $-p$                      & $p^2-p-1 + p\sum_{x(p)}\js{4x^3+x^2+2x+1}$         \\ \hline
$y^2=x^3+tx^2+tx+t^2$    & 1    & $-p$                      & $p^2-p-1-\delta_1(p) \cdot 2p$                                             \\ \hline
\end{tabular}%
}
\caption{\label{table:michelleone} First and second moments of some families from \cite{Wu}.}
\end{table}

These families provide additional support for the Bias Conjecture. It is worth remembering that these families were all carefully chosen so that the resulting sums could be computed exactly; as this is not the case in general, it is quite likely that the behavior here is not representative.

As a further test, two-parameter families were studied. Unfortunately the Rosen-Silverman theorem is not available to determine the rank, and we just record the values of the first moments as for our purposes we do not need to compute the ranks (though see \cite{Waz} for results on ranks). We report the results from \cite{Wu} in Table \ref{table:michelletwo}, with $\delta_1(p)$ as before (it is $1$ if $p \equiv 1 \bmod 4$ and 0 otherwise) and $\delta_3(p) = 1$ if $p \equiv 3 \bmod 4$ and $0$ otherwise.

\begin{table}[htb]
\resizebox{\textwidth}{!}{%
\begin{tabular}{|l|l|l|l|}
\hline
Two-Parameter Family    & $p^2A_{1,\mathcal{F}(p)}$ & $p^2A_{2,\mathcal{F}(p)}$                               \\ \hline
$y^2=x^3+tx+sx^2$        & 0                      & $p^3-2p^2+p$                                                \\ \hline
$y^2=x^3+t^2x+st^4$         & 0                      & $p^3-2p^2+p-2(p^2-p)\js{-3}$ \\ \hline
$y^2=x^3+sx^2-t^2x$          & 0                      & $p^3 - p^2 - \delta_3(p) (2p^2 - 2p)$ \\ \hline
$y^2=x^3+ts^2x^2+(t^3-t^2)x$       & $-p^2$                     &
    $p^3-3p^2+3p-1-\delta_3(p) (2p-2)$                               \\
\hline
$y^2=x^3+t^2x^2+(t^3-t^2)sx$        & $-p^2$                    & $p^3 - 3p^2 + 3p - \delta_3(p) (2p^2 -4p)$                     \\ \hline
$y^2=x^3+t^2x^2-(s^2-s)t^2x$        & $-2p^2$                     & $p^3 - 3p^2 +2p + \delta_1(p) (p-\sum_{s(p)}\sum_{x,y(p)}\js{x^3-(s^2-s)x}\js{y^3-(s^2-s)y})$\\ \hline
$y^2=x^3-t^2x+t^3s^2+t^4$        & $-2p^2$                     & $p^3-2p^2+p-\left[\js{-3}+\js{3}\right]p^2$                      \\
\hline
\end{tabular}%
}
\caption{\label{table:michelletwo} First and second moments from some two-parameter families from \cite{Wu}. }
\end{table}

We thus again see a negative bias in the largest lower order term in the second-moment expansion. This should be understood in the sense of a normalized average over primes $p$, e.g. as in the introduction and elsewhere in the paper. For example, consider the family $y^2=x^3-t^2x+t^3s^2+t^4$ appearing in the last row of Table \ref{table:michelletwo}. The largest lower order term of the second moment $p^2A_{2,\mathcal{F}(p)}$ is on average negative in the sense that
    \[
    \lim_{X \rightarrow \infty} \frac{1}{\pi(X)} \sum_{p \leq X} \frac{p^2A_{2,\mathcal{F}(p)} - p^3}{p^2} = -2 < 0.
    \]
To make this calculation, compute that for $p\geq 5$ we have
    \[
    \js{-3}+\js{3} = \begin{cases} 2 \quad \quad \text{if $p \equiv 1\pmod{12}$} \\ 0 \quad  \quad \text{if $p \equiv 7, 11\pmod{12}$} \\ -2 \quad \text{ if $p \equiv 5 \pmod{12}$} \end{cases}
    \]
and moreover note that $p$ equidistributes among the congruence classes $1,5,7,11$ modulo $12$ as $X \rightarrow \infty$. When $p \equiv 5 \pmod{12}$ we have $p^2A_{2,\mathcal{F}(p)} = p^3+p$ and the largest lower order term $+p$ is positive. However, such terms do not contribute after dividing by $p^2$ and averaging over $p$. Instead, the lower order terms of size $p^2$ coming from the cases $p \equiv 1,7,11 \pmod{12}$ dictate the average. A similar interpretation should be applied to the conclusion of Lemma \ref{lemma:twoParameterBias} below.

We now give the details for a representative calculation, that of the family $y^2=x^3+t^2x+st^4$. Note that we must adjust our definitions of $A_{r,\mathcal{F}}$ and divide by $p^2$ and not $p$, as we have sums over $t$ and $s$ modulo $p$.

\begin{lemma}
For $p \geq 5$, the first moment of the two-parameter family $y^2=x^3+t^2x+st^4$ is $0$.
\end{lemma}

\begin{proof} The claim follows from straightforward algebra; briefly we note there is no contribution from $t=0$, and then we can send $s$ to $t^{-1}s \bmod p$, and the resulting $s$ sum is complete and thus zero. Explicitly,
\begin{eqnarray}
    -p^2A_{1,\mathcal{F}(p)} & \ = \ & -\sum_{t(p)}\sum_{s(p)}a_{t,s}(p) =\sum_{t(p)}\sum_{x(p)}\sum_{s(p)}\js{x^3+t^2x+st^4} \nonumber\\
    & = & \sum_{t=1}^{p-1}\sum_{x(p)}\sum_{s(p)}\js{t^3x^3+t^3x+st^4}\nonumber\\
    & = & \sum_{t=1}^{p-1}\sum_{x(p)}\sum_{s(p)}\js{t^3}\js{x^3+x+st}\nonumber\\
    & = & \sum_{x(p)}\sum_{s(p)}\sum_{t=1}^{p-1}\js{t}\js{st+(x^3+x)}\nonumber\\
    & = & \sum_{x(p)}\sum_{s(p)}\sum_{t=1}^{p-1}\js{t}\js{t^{-1}st+(x^3+x)}\nonumber\\
    & = & \sum_{x(p)}\sum_{s(p)}\sum_{t=1}^{p-1}\js{t}\js{s+(x^3+x)}.
\end{eqnarray}
\end{proof}

\begin{lemma}\label{lemma:twoParameterBias} For $p \ge 5$ the second moment times $p^2$ of the two-parameter family $y^2=x^3+t^2x+st^4$ is $p^3-2p^2+p-2(p^2-p)\js{-3}$, which supports the Bias Conjecture.
\end{lemma}

\begin{proof} We see there is no contribution from $t=0$, we then send $x$ and $y$ to $tx$ and $ty$ modulo $p$ pull out a $t^6$ from the Legendre sums, and then send $s$ to $t^{-1}s \bmod p$ to simplify further:
   \begin{eqnarray}
    p^2A_{2,\mathcal{F}(p)} & \ = \ &  \sum_{t,s(p)}{a_{t,s}}^2(p) \ = \   \sum_{t(p)}\sum_{s(p)}\sum_{x,y(p)}\js{x^3+t^2x+st^4}\js{y^3+t^2y+st^4} \nonumber\\
    & \ =\ & \sum_{t=1}^{p-1}\sum_{s(p)}\sum_{x,y(p)}\js{t^3x^3+t^3x+st^4}\js{t^3y^3+t^3y+st^4} \nonumber\\
    & \ = \ &   \sum_{t=1}^{p-1}\sum_{s(p)}\sum_{x,y(p)}\js{t^6}\js{x^3+x+st}\js{y^3+y+st}\nonumber\\
    & \ = \ &  \sum_{t=1}^{p-1}\sum_{s(p)}\sum_{x,y(p)} \js{s+(x^3+x)}\js{s+(y^3+y)} \nonumber\\
    & \ = \ & (p-1) \sum_{x,y(p)} \sum_{s(p)}  \js{s+(x^3+x)}\js{s+(y^3+y)}.
\end{eqnarray}

To evaluate this sum, consider the quadratic polynomial in $s$ given by $(s+(x^3+x))(s+(y^3+y))$, and write $\delta(x,y)$ for the discriminant of this quadratic. The second part of Lemma \ref{lem:sumlegendrelinquad} gives
    \[
    \sum_{s(p)}  \js{s+(x^3+x)}\js{s+(y^3+y)} \ = \  \begin{cases} -1 & \mbox{{\rm if} } p \nmid \delta(x,y) \\ (p - 1) & \mbox{{\rm if} } p \mid \delta(x,y). \end{cases}
    \]
This gives
\begin{align}
    p^2A_{2,\mathcal{F}(p)} \ = \  (p-1)\left[p\left(\sum_{x,y \bmod p \atop \delta(x,y)\equiv 0(p)} 1 \right) \ - \ p^2\right].
\end{align}
We have thus reduced the problem to determining how often $p$ divides \be \delta(x,y) \ = \ [(x-y)(y^2+xy+(1+x^2))]^2. \ee Set $\delta_1(x,y) = x-y$ and $\delta_2(x,y) = y^2+xy+(1+x^2)$, so that $\delta(x,y) = \delta_1(x,y)^2\delta_2(x,y)^2$. For each $\delta_i$, we count the number of pairs $(x,y)$ such that $\delta_i(x,y)$ is zero modulo $p$, and then subtract the doubly counted pairs where both factors vanish modulo $p$.\\ \

Case 1: We have $\delta_1(x,y)=x-y\equiv 0 \bmod p$ when $x=y$; thus there are $p$ such pairs. \\ \

Case 2: We need to count the number of solutions of $\delta_2(x,y)=y^2+xy+(1+x^2) \equiv 0 \bmod p$. By the Quadratic Formula modulo $p$, we have
\begin{eqnarray}
    y \ = \ \frac{-x \pm \sqrt{-3x^2-4}}{2}, \nonumber
\end{eqnarray} and thus there are two solutions if $-3x^2-4$ is a non-zero square, one solution if it is zero modulo $p$, and otherwise no solutions.\footnote{This is why we assumed $p > 2$, so 2 is invertible modulo $p$.} Thus, summing over $x$ yields the number of pairs is
\begin{eqnarray}
\sum_{x(p)}\left[1+\js{-3x^2-4}\right] \ = \ p + \sum_{x(p)} \js{-3x^2-4} \ =  \ p-\js{-3},
\end{eqnarray} where we used Lemma \ref{lem:sumlegendrelinquad} to evaluate the quadratic sum (the discriminant is -48, which is invertible modulo $p$ as we assumed $p \ge 5$).\\ \

Case 3: The double-counted pairs satisfy both $x=y$ and $y^2+xy+(1+x^2)\equiv 0 \bmod p$, which means that they satisfy $3x^2+1\equiv0 \bmod p$, or $-3x^2\equiv 1 \bmod p$. Thus, there is a double-counted solution if and only if $\js{-3}=1$, and the number of double-counted pairs is $1+\js{-3}$.

\bigskip

Therefore, the total number of pairs for $\delta(x,y)\equiv0 \bmod p$ is \be p + \left[p-\js{-3}\right] - \left[1 + \js{-3}\right] \ = \  2p-1-2\js{-3}. \ee
Hence, the second moment times $p^2$ of the family equals
\begin{eqnarray}
    p^2A_{2,\mathcal{F}(p)}
    &\ =\ & (p-1)\left[p\left(2p-1-2\js{-3}\right) - p^2\right] \nonumber\\
    & = & p^3-2p^2+p-2(p^2-p)\js{-3}.
\end{eqnarray}
\end{proof}

\section{Biases in Higher Moments}\label{sec:rogerwang}

\noindent By Steven J. Miller and Yan Weng \\ \

We end with some additional results and directions for future research.

Miller and Weng recently pushed these investigations in two new directions; below we briefly summarize their findings, and refer the reader to \cite{MWe, MWeWu} for details. First, they studied a large number of families of elliptic curves where it was not clear if there was a closed form polynomial expansion for the second moment (previous families to date were chosen because the resulting sums could be done in closed form, and thus these may not be indicative of the behavior of a generic family). By looking at the values for primes in different congruence classes, they were able to conjecture formulas for second moments and in some cases prove them, but for a typical family they found that the formula cannot be a polynomial in $p \bmod 2^a 3^b 5^c$ (unless $a, b$ or $c$ are quite large!). This is not  surprising, as there are a few families where the second moments can be explicitly computed and are not polynomials.

Next, they looked at families with moderate rank to see if the bias is still present; all the families that can be analyzed in closed form have small rank and polynomials in the defining equations of low degree. While such families are computationally difficult to deal with for many quantities (for example, the analytic conductors grow very rapidly, making it hard to compute more than a few low-lying zeros), the bias calculations are not significantly harder than what has been done before. The reason is that the Legendre sums only involve the polynomials modulo the prime, and thus the large coefficients are immediately reduced. The difficulty is that in these families it seems that the $p^{3/2}$ term is frequently present, making it hard to see a bias in terms of size $p$. The data is not inconsistent with a \emph{positive} bias for these higher rank families, indicating further study is warranted.

Their final experiment was to examine if there is a bias in the higher moments. Using the theory of modular forms, one can write down formulas involving averages of powers of the Satake parameters, but unfortunately these cannot be evaluated in closed form. Computationally, it is very difficult to see numerical biases for these moments as the conjecture states that the first lower order term that does not average to zero has a negative average; if there is a larger lower order term that does average to zero, it will drown out this bias and make it impossible to detect numerically.

Below is a summary of their main results.

\begin{itemize}




\item All the rank $0$ and rank $1$ families studied have data consistent with a negative bias in their second moment sums. However, for higher rank families (rank at least 4) the data suggests that there is instead a positive bias.

\item For the fourth moment, we also believe that the rank $0$ and rank $1$ families have negative biases. We see this in some families; in others we
see the presence of terms of size $p^{5/2}$ whose behavior is consistent with their averaging to zero, but its presence makes it impossible to detect the lower order terms. Interestingly, again for higher rank families (rank at least 4) the data is consistent with a positive bias.

\item The sixth moment results are similar to the fourth moment. The results are consistent with either a negative bias, or a leading term (now of
size $p^{7/2}$) averaging to zero for lower rank families. Our data also suggests that higher rank families have positive biases.

\item For the odd moments, the coefficients of the leading term vary with the primes. Our data suggests that the average value of the main term for
the $(2k+1)$\textsuperscript{st} moment is $-C_{k+1}rp^{k+1}$, where $C_n = \frac1{n+1}\ncr{2n}{n}$ is the $n$\textsuperscript{th} Catalan number.
\end{itemize}


\ \\

\end{document}